\documentclass[11pt]{amsart}
\usepackage{amssymb}
\usepackage{amssymb, amsmath}
\usepackage{verbatim}
\usepackage{framed}
\numberwithin{equation}{section}
\let\d=\partial

\let\wt=\widetilde
\let\wh=\widehat

\def\cB{{\mathcal B}}
\def\cC{{\mathcal C}}

\def\N{{\mathbb N}}
\def\R{{\mathbb R}}
\def\T{{\mathbb T}}
\def\Z{{\mathbb Z}}

\def\ep{\varepsilon}

\def\virgp{\raise 2pt\hbox{,}}
\def\cdotpv{\raise 2pt\hbox{;}}
\def\Id{\mathop{\rm Id}\nolimits}

\def\div{ \hbox{\rm div}\,  }

\def\divy{\, \hbox{\rm div}_y\,  }
\def\divu{\, \hbox{\rm div}_u\,  }

\newcommand{\Int}{\displaystyle \int}

\def\du{\delta\!u}
\def\dQ{\delta\!Q}
\def\dA{\delta\!A}

\newtheorem{thm}{Theorem}[section]
\newtheorem{lem}{Lemma}[section]
\newtheorem{rmk}{Remark}[section]
\newtheorem{cor}{Corollary}[section]
\newtheorem{prop}{Proposition}[section]

\newcommand{\ben}{\begin{eqnarray}}
\newcommand{\een}{\end{eqnarray}}
\newcommand{\beno}{\begin{eqnarray*}}
\newcommand{\eeno}{\end{eqnarray*}}
\begin{document}
\title[]{The incompressible Navier-Stokes equations in vacuum}
\author[R. Danchin]{Rapha\"{e}l Danchin}
\address[R. Danchin]{Universit\'{e} Paris-Est,  LAMA (UMR 8050), UPEMLV, UPEC, CNRS,
 61 avenue du G\'{e}n\'{e}ral de Gaulle, 94010 Cr\'{e}teil Cedex, France.} \email{raphael.danchin@u-pec.fr}
\author[P.B. Mucha]{Piotr Bogus\l aw Mucha}
\address[P.B. Mucha]{Instytut Matematyki Stosowanej i Mechaniki,
 Uniwersytet Wars\-zawski, 
ul. Banacha 2,  02-097 Warszawa, Poland.} 
\email{p.mucha@mimuw.edu.pl}

\begin{abstract}  We are concerned with  the existence and uniqueness issue  
for  the inhomogeneous incompressible Navier-Stokes equations supplemented with  $H^1$ initial velocity 
and \emph{only bounded nonnegative density}. In contrast with all the previous works on 
that topics, we do not require  regularity or positive
lower bound for the initial density, or  compatibility conditions for  the initial velocity, and still obtain \emph{unique} solutions. 
Those solutions are  global in the two-dimensional  case for general  data, and in the three-dimensional case if the velocity satisfies a 
suitable scaling invariant smallness condition.
As a straightforward application, we provide  a complete answer to Lions' question in \cite{PLL}, page 34, concerning the
evolution of a drop of incompressible viscous fluid in the vacuum. 
\end{abstract}
\maketitle

\section{Introduction}


Since the pioneering works by Lichtenstein \cite{lichtenstein}, Wolibner \cite{W} and Leray \cite{Leray} at the beginning 
of the XXth century, studying  fluid mechanics  models  has  generated
important advances in the development of mathematical analysis. Very schematically, 
classical fluid mechanics  is divided into two types of models  corresponding to whether the fluid is
homogeneous or not. 
 On the one hand, the  incompressible Navier-Stokes equations  
 $$ \begin{array}{lcl}
 v_t +v \cdot \nabla v - \mu\Delta v +\nabla P =0 & \mbox{in} & \R_+\times\Omega,\\[1ex]
\div v=0 & \mbox{in} &  \R_+\times\Omega,\\
 \end{array}\leqno(NS)$$ 
  govern the evolution of   the velocity field $v=v(t,x)\in\R^d$ 
 and pressure function   $P=P(t,x)\in\R$ 
 of a homogeneous  incompressible viscous fluid with constant viscosity $\mu>0$
  (here  $t\geq 0$ stands for the time variable  and $x\in \Omega,$ for the position in the fluid domain  $\Omega\subset\R^d$).
On the other hand,  the evolution of compressible viscous flows obeys the following system 
$$
 \begin{array}{lcl}
  \rho_t+ \div(\rho v) =0 & \mbox{in} & \R_+\times\Omega, \\[1ex]
\rho v_t +\rho v \cdot \nabla v - \mu\Delta v -\mu'\nabla\div v +\nabla P =0 & \mbox{in} &  \R_+\times\Omega,
 \end{array}\leqno(CNS)
 $$
where  $\rho=\rho(t,x)\geq0$ stands for the density of the fluid and $P=P(\rho)$ is a given 
pressure function. 
In between $(NS)$ and $(CNS),$   we find the \emph{inhomogeneous} Navier-Stokes 
system  that governs the evolution of incompressible viscous flows with \emph{nonconstant} density.
That system   founds his place in the theory of geophysical flows, where
 fluids are incompressible but with variable density, like in  oceans or rivers. 
In the present paper, we are concerned with that latter  system that reads:
$$
 \begin{array}{lcl}
  \rho_t+v \cdot \nabla \rho =0 & \mbox{in} & \R_+\times\Omega, \\[1ex]
\rho v_t +\rho v \cdot \nabla v - \mu\Delta v +\nabla P =0 & \mbox{in} &  \R_+\times\Omega,\\[1ex]
\div v=0 & \mbox{in} &  \R_+\times\Omega.
 \end{array}\leqno(INS)
 $$
The unknowns are  the velocity field $v=v(t,x),$  the  density $\rho=\rho(t,x)$ and  the pressure $P=P(t,x).$
We shall  assume that the fluid domain $\Omega$ is either 
the   torus $\T^d$ (that is the fluid domain is $]0,1[^d$ and  $(INS)$ is supplemented with periodic boundary conditions), 
or a $\cC^2$ simply connected bounded domain of $\R^d$ with $d=2,3.$ 
In that latter case, System   $(INS)$ is supplemented  with homogeneous 
Dirichlet boundary conditions for the velocity.   

\medbreak
It is well known that sufficiently smooth solutions to $(INS)$
fulfill   for all $t\geq0$:
\begin{itemize}
\item The energy balance:
\begin{equation}\label{eq:ene}
 \frac 12 \frac{d}{dt} \int_{\Omega} \rho|v|^2 \,dx +\mu \int_{\Omega} |\nabla v|^2\, dx =0.
\end{equation}
\item The  conservation of total momentum (in the case $\Omega=\T^d$): 
\begin{equation}\label{eq:m}
 \int_{\Omega} (\rho v)(t,x)\, dx=\int_{\Omega} (\rho_0 v_0)(x) \,dx.
\end{equation}
\item  The conservation of total mass: 
\begin{equation}\label{eq:M}
\int_{\Omega}\rho(t,x)\,dx=\int_{\Omega}\rho_0(x)\,dx.
\end{equation}
\item Any Lebesgue norm of $\rho_0$ is preserved through the evolution, and 
\begin{equation}\label{eq:infsup}
\inf_{x\in\Omega} \rho(t,x)=\inf_{x\in\Omega} \rho_0(x)\quad\hbox{and}\quad
\sup_{x\in\Omega} \rho(t,x)=\sup_{x\in\Omega} \rho_0(x).
\end{equation}
\end{itemize}

 The constant density case, that is System  $(NS),$ has  been intensively 
investigated for the  last 80 years. Since  the   works by J. Leray  \cite{Leray} in 1934
and  O. Ladyzhenskaya \cite{Lad} in 1959 (see also \cite{LP}),   it is known that:
\begin{itemize}
\item In dimension $d=2,3$,  for any $v_0$ in $L_2(\Omega)$ with $\div v_0=0,$ there exist global weak solutions 
 (the so-called turbulent solutions)  to $(NS),$ satisfying 
$$
\|v(t)\|_2^2+2\mu\int_0^t\|\nabla v\|_2^2\,d\tau\leq  \|v_0\|_2^2.
$$
\item In the 2D case, turbulent solutions are unique, and 
additional regularity is preserved. In particular, if $v_0$ is in  $H^1_0(\Omega)$, then  $v\in\cC_b(\R_+;H^1_0(\Omega))$.
\item In the 3D case, if  in addition $v_0$ in $H^1_0(\Omega)$ and 
\begin{equation}\label{eq:small-vit}
\mu^{-2}\|v_0\|_2\|\nabla v_0\|_2\quad\hbox{is small enough}
\end{equation} then there exists a unique   global solution    $(v,\nabla P)$ with $v$ in $\cC_b(\R_+;H^1_0(\Omega))$.
\end{itemize}
For a large three-dimensional $v_0$  in $H_0^1(\Omega),$  we have  a unique local-in-time smooth solution, 
but  proving that smoothness persists for all time  is essentially  the  global regularity issue  of one of the Millennium problems (see {\it http://www.claymath.org/millennium-problems}). 
\medbreak
As regards  the inhomogeneous Navier-Stokes equations,   the state-of-the-art says that  the weak solution theory is similar to  the one of the  homogeneous case, 
and so is the strong solution theory if, beside smoothness,   \emph{the density is bounded away from zero}. 
More precisely, the following results are available: 
\begin{itemize}
\item \emph{Global weak solutions with finite energy:}  If $d=2,3,$ 
whenever $0\leq\rho_0\leq\rho^*$  for some $\rho^*>0,$ and $\sqrt{\rho_0}\,v_0$ is in $L_2,$ there exists  a global distributional solution
$(\rho,v,P)$ to $(INS)$ satisfying \eqref{eq:M}, and  such that for all $t\geq0,$
\begin{equation}\label{eq:uL2}
\|\sqrt{\rho(t)}\,v(t)\|_2^2+2\int_0^t\|\nabla v\|_2^2\,d\tau\leq  \|\sqrt{\rho_0}\,v_0\|_2^2,\qquad
\quad\hbox{and}\quad 0\leq \rho(t)\leq\rho^*.
\end{equation}
The constant viscosity case with  $\inf\rho_0>0$ has been solved by   A. Kazhikhov \cite{K}, 
then J. Simon \cite{S} removed the lower bound assumption on  $\rho_0,$ and  
P.-L. Lions \cite{PLL} proved that $\rho$ is in fact a  renormalized solution of the mass equation, 
which enabled him to consider also the case where  $\mu$ depends on $\rho$ (see also \cite{Des-DIE}). \smallbreak 
\item \emph{Global strong solutions in the 2D case:}
They have been first constructed by O. Ladyzhenskaya and V. Solonnikov in \cite{LS} in the bounded domain case, 
whenever   $v_0$ is in $H^1_0(\Omega)$ 
and $\rho_0$ is in $W^1_\infty$ with  $\inf\rho_0>0.$
\smallbreak
\item \emph{Strong solutions in the 3D case:}  Under the hypotheses of the 2D case, 
there exists a unique local-in-time maximal strong solution, 
and if  $v_0$ is small enough, then that solution is global (see \cite{LS}).  
\end{itemize}
\smallbreak

After the work of O. Ladyzhenskaya and V. Solonnikov \cite{LS}, 
a number of   papers have been devoted to the study of  strong solutions to $(INS)$
and, more particularly,  to  classes of data generating  regular unique solutions. 
Recent developments involve  two directions:
\begin{itemize}
\item \underline{Finding  minimum  assumptions for  uniqueness  in the nonvacuum case}: 
 Here one can mention the critical  regularity approach of \cite{Dan1, Dan2}
 where  density has to be continuous, bounded and bounded away from $0,$ 
 and relatively new  works  like    \cite{DM-cpam, DM-arma} (further improved in  \cite{CHW,DZP,HPZ, PZZ}), 
   relying on the use of Lagrangian coordinates, and  where the density
   need not be continuous.  Recently, in connection with Lions' question, lots of  attention has been brought to the case where the initial 
   density is given by    $$   \rho_0=\eta_1 1_{D_0}+\eta_2 1_{{}^c\!D_0},\qquad \eta_1,\eta_2>0,\qquad  D_0\subset\Omega.$$
   The main issue is whether the smoothness 
   of $D_0$ is preserved through the time evolution (see \cite{DZ,GGJ,LZ1,LZ2}).
\item \underline{Smooth data with allowance of vacuum}: 
As pointed out in \cite{Cho-Kim} (see also \cite{CHW}, and \cite{Germain} as regards the weak-strong uniqueness issue), one
can solve $(INS)$ uniquely  in presence of vacuum 
if $\rho_0$ is smooth enough \emph{and $v_0$ satisfies  the compatibility condition} 
\begin{equation}\label{eq:comp}
-\Delta v_0+\nabla P_0=\sqrt{\rho_0}\,g\quad\hbox{for some }\ 
g\in L_2(\Omega)\ \hbox{ and }\ P_0\in H^1(\Omega).
\end{equation} 
Very recently, in \cite{Li},  J. Li   pointed out that  Condition \eqref{eq:comp}
can be  removed  but, still, the density must be regular and  only local-in-time solutions are produced.  
\end{itemize}

If  it is not assumed that the density is bounded away from zero then the analysis  gets wilder, since the system degenerates (in vacuum regions, the term $\rho v_t$ in the momentum equation vanishes),
 and the general strong solution theory is still  open,  even in the 2D case. 
Our main goal is to show existence and uniqueness  for $(INS)$ supplemented with general  initial data satisfying 
 the following `minimal' assumptions:
\begin{equation}\label{id}
 v_0\in H^1_0(\Omega)\ \hbox{ with }\ \div v_0=0,  \mbox{ \ \ and \ \ }  0 \leq \rho_0 \leq \rho^*<+\infty.
\end{equation}
In other words, we  aim  at completing  the program initiated by J. Leray, 
\emph{for general inhomogeneous fluids with  just bounded initial density}, establishing that 
\begin{itemize}
 \item in the 2D case, for arbitrary initial data fulfilling just \eqref{id} there exists
 a  \underline{unique}  global-in-time solution with regular velocity;
  \item in the 3D case: for arbitrary initial data fulfilling  \eqref{id} there exist a 
   \underline{unique}   local-in-time  solution with regular velocity, that  is global  if
    \eqref{eq:small-vit} is fulfilled. 
\end{itemize}

Let us emphasize that, since  we do not require any regularity or positive lower bound for the density,  one can consider 
 `patches of density', that is  initial densities that are   characteristic functions of subsets of $\Omega$
 (this corresponds for instance to a drop of incompressible fluid in  vacuum or the opposite: a bubble of
 vacuum embedded in the fluid). As a consequence of our results, we  show   the persistence of  the interface  regularity 
 through the evolution, which constitutes a complete answer to  the question  raised by P.-L. Lions in \cite{PLL} at page 34.
    \smallbreak
    In order to prove the above existence and uniqueness results for $(INS)$ supplemented with 
    data satisfying just \eqref{id}, the main difficulty  is to propagate enough regularity for the velocity
 to ensure uniqueness, while  the density is rough and likely to vanish in some 
 parts of the fluid domain.  Recall that in most   evolutionary fluid mechanics models, 
the uniqueness issue is closely connected to the Lipschitz control of the flow of 
the velocity field $v$, hence to the fact  that $\nabla v$ is in $L_1(0,T;L_\infty(\Omega)).$ 
  The main breakthrough of our paper is that we  manage to 
 keep  the $H^1$ norm of the velocity under control for all time \emph{despite vacuum} and 
to  exhibit a parabolic gain of regularity  which is slightly weaker than the standard one, 
but still sufficient to eventually get   $\nabla v$ in $L_1(0,T;L_\infty(\Omega)).$ 
To achieve it, we   combine time weighted energy estimates in the spirit of  \cite{Li}, classical Sobolev embedding and  {\it shift of integrability} 
from time variable to space variable (more details are provided in the next section). 
\smallbreak
   We  shall assume throughout that the fluid  domain is either the  torus $\T^d$  or  a $\cC^2$   bounded domain of $\R^d,$
        the generalization to unbounded domains (even the whole space) 
   within our approach  being unclear  as regards global-in-time results.  
For simplicity, we  only consider $H^1$ initial velocity fields, even though it should be possible to have  less
regular data, like in \cite{PZZ}.  As regards the domain $\Omega,$ we do not strive for minimal regularity assumptions either.     
     \medbreak
   We follow the standard notation for the evolutionary PDEs.  By $\nabla$ we denote the gradient with respect to space variables, and by $\partial_t u$ or $u_t,$ the time derivative of function $u$.
   By $\|\cdot\|_p,$ we mean $p$-power Lebesgue norms over $\Omega$; $L_p(Q)$ is 
   the $p$-Lebesgue space over a set $Q$; we denote  by $H^s$ and $W^s_p$  the Sobolev
   (Slobodeckij for $s$ not integer) space,  and put $H^s=W^s_2.$
Generic constants are denoted by $C,$ and $A\lesssim B$ means that $A\leq CB.$

Finally, as a great part of our analysis will concern $H^1$ regularity and will work indistinctly in
a bounded domain or in the torus, we shall adopt slightly abusively the notation $H^1_0(\Omega)$
to designate the set of $H^1(\Omega)$ functions that vanish at the boundary if $\Omega$ is a bounded domain, 
or general $H^1(\T^d)$ functions if $\Omega=\T^d.$ 
   \medbreak
The rest of the paper unfolds as follows. In the next section, we state the main results. 
In Section \ref{s:existence}, we concentrate on a priori estimates and on  the proof of existence for $(INS)$
while  Section \ref{s:uniqueness} concerns the uniqueness  issue. 
A few technical results (in particular a key logarithmic interpolation inequality)
are postponed in the appendix.


\section{Results}

Here we   state of our  results and give a overview of the strategy   that we used to achieve them.
Let us first write out our main result in the two dimensional case.
 \begin{framed}
\begin{thm}\label{thm:d=2}
Let $\Omega$ be a $\cC^2$ bounded subset of $\R^2,$ or the torus $\T^2.$
Consider any  data $(\rho_0,v_0)$  in $L_\infty(\Omega)\times H^1_0(\Omega)$  satisfying  for some constant $\rho^*>0,$ 
 \begin{equation}\label{eq:data}
  0\leq \rho_0\leq\rho^*,\quad  \div v_0=0 
  \mbox{ \ \ and \ \ }    M:= \int_{\Omega} \rho_0\, dx >0. 
 \end{equation}
 Then System $(INS)$ supplemented with data $(\rho_0,v_0)$ admits 
a unique global solution $(\rho,v,\nabla P)$ satisfying  \eqref{eq:ene}, \eqref{eq:m} (in the case $\Omega=\T^d$), \eqref{eq:M}, \eqref{eq:infsup}, 
 the following properties of regularity:
$$\rho\in L_\infty(\R_+;L_\infty(\Omega)),\quad v \in L_\infty(\R_+;H^1_0(\Omega)),\quad 
\sqrt{\rho}v_t, \nabla^2 v,\nabla P \in L_2(\R_+;L_2(\Omega))$$
and also, for all $1\leq r<2$ and $1\leq m<\infty,$
$$\nabla(\sqrt t P),\,\nabla^2(\sqrt t v) \in L_\infty(0,T;L_r(\Omega))\cap L_2(0,T;L_m(\Omega))\quad\hbox{for all } T>0.$$
Furthermore,  we have   $\sqrt\rho v\in \cC(\R_+;L_2(\Omega)),$
$\rho\in\cC(\R_+;L_p(\Omega))$ for all finite $p,$ and   $v \in H^{\eta}(0,T;L_p(\Omega))$ for all $\eta<1/2$  and $T>0.$
\end{thm}
\end{framed}
In the three  dimensional case we have:
\begin{framed} \begin{thm}\label{thm:d=3}
Let $\Omega$ be a $\cC^2$ bounded subset of $\R^3$ or the torus $\T^3.$
There exists a  constant $c>0$ such that 
for  any  data $(\rho_0,v_0)$  in $L_\infty(\Omega)\times H^1_0(\Omega)$  satisfying  
\eqref{eq:data} and 
\begin{equation}\label{eq:smalld=3}
(\rho^*)^{\frac32}\|\sqrt{\rho_0}\, v_0\|_{2}\|\nabla v_0\|_2\leq c\mu^2, 
\end{equation}
 System $(INS)$ supplemented with data $(\rho_0,v_0)$ admits 
a unique global solution $(\rho,v,\nabla P)$ satisfying Identities \eqref{eq:ene}, \eqref{eq:m}  (in the case $\Omega=\T^d$), \eqref{eq:M}, \eqref{eq:infsup}, and
such that
$$\displaylines{\rho\in L_\infty(\R_+;L_\infty(\Omega)),\quad v \in L_\infty(\R_+;H^1_0(\Omega)),\quad 
\sqrt{\rho}v_t, \nabla^2 v,\nabla P \in L_2(\R_+;L_2(\Omega))\cr
\hbox{and}\quad\nabla(\sqrt t P),\,\nabla^2(\sqrt t v) \in L_\infty(0,T;L_2(\Omega))\cap L_2(0,T;L_6(\Omega))\quad\hbox{for all } T>0.}$$
Furthermore,  we have   $\sqrt\rho v\in \cC(\R_+;L_2(\Omega)),$
$\rho\in\cC(\R_+;L_p(\Omega))$ for all finite $p,$ and   $v \in H^{\eta}(0,T;L_6(\Omega))$ for all $\eta<1/2$  and $T>0.$
\end{thm}
\end{framed}

As a by-product, we obtain the following  answer to  Lions' question (\cite{PLL}, page 34):
\begin{framed}
\begin{cor}\label{cor:drop} Let $\Omega$ be a $\cC^2$ bounded domain of $\R^d$  (or the torus $\T^d$) with $d=2,3.$ 
Assume that  $\rho_0= 1_{D_0}$ for some open subset $D_0$ of $\Omega$ with $\cC^{1,\alpha}$ regularity ($0<\alpha<1$ if $d=2$ and $0<\alpha<1/2$ if $d=3$). 
Then for any divergence free initial velocity $v_0$ in  $H^1_0(\Omega)$ (satisfying \eqref{eq:smalld=3} with $\rho^*=1$ if $d=3$), 
the unique global solution $(\rho,v,\nabla P)$ provided by the above theorems is such that   for all $t\geq0,$
$$\rho(t,\cdot)=1_{D_t} \ \hbox{ with }\ D_t := X(t,D_0),$$ where   $X(t,\cdot)$ stands for the flow of $v,$  that is the unique solution of 
\begin{equation}\label{l1}
 \frac{dX}{dt}= v(t,X), \qquad X|_{t=0}=y,\qquad y\in\Omega.
\end{equation}
Furthermore,  $D_t$ has  $\cC^{1,\alpha}$ regularity with a control of the H\"older
norm in terms of the initial data. 
\end{cor}\end{framed}
\begin{proof}
Assume that $D_0$ corresponds  to the level set $\{f_0=0\}$ of some function 
$f_0:\Omega\to\R$ with $\cC^{1,\alpha}$ regularity. Then  we have  $D_t:=f_t^{-1}(\{0\})$ 
with $f_t:= f\circ X(t,\cdot).$ Fix some $T>0.$ 

In the 2D case,     Theorem \ref{thm:d=2} and interpolation imply that we have
$$\sqrt t\nabla^2 v\in L_{2+\ep}(0,T;L_{1/\ep}(\Omega))\ \hbox{ for all small enough }\ep>0.$$ By Sobolev embedding with 
respect to the space variable, and H\"older inequality with respect to the time variable, one can conclude 
 that    $\nabla v$ is in $L^1(0,T;\cC^{0,\beta})$ for all $\beta\in(0,1).$
Consequently the flow $X(t,\cdot)$  is in $\cC^{1,\beta}$ for all $\beta\in(0,1),$
which implies that $f_t$ is in $\cC^{1,\alpha}$ provided that $\alpha<1.$

Similarly, in the 3D case, Theorem  \ref{thm:d=3}
ensures that $\sqrt t\nabla^2 v$ is in $L_{2+\ep}(0,T;L_{6-\ep}(\Omega)),$
and thus  $\nabla v$ is in $L_{1}(0,T;W^1_r)$ for all $r<6.$ 
This implies that  the flow $X(t,\cdot)$ is in $\cC^{1,\beta}$  for all $\beta\in (0,1/2),$
and thus $f_t$ is in $\cC^{1,\alpha}$ if $\alpha<1/2.$ \end{proof}
\begin{rmk}\label{r:local} In the 3D case,  there exists a  constant $c=c(\Omega)$ such that if Condition \eqref{eq:smalld=3} is not satisfied, 
then Theorem \ref{thm:d=3} and Corollary \ref{cor:drop} hold true on the time interval $[0,T]$ with 
 $\displaystyle T:=\Bigl(\frac{\mu}{\rho^*}\Bigr)^7\frac{c\rho^*}{\|\sqrt\rho_0 v_0\|_2^2\|\nabla v_0\|_2^6}\cdotp$
\end{rmk}

\begin{rmk} The time regularity issue is rather subtile. 
In fact, unless the density is bounded away from $0,$
we do not have $v$ in $\cC(\R_+;L_2).$ 
At the same time, as the kinetic energy satisfies the energy balance \eqref{eq:ene}, 
we do have  $\sqrt{\rho} v \in \cC(\R_+;L_2).$
It follows that, whenever the initial kinetic energy vanishes (that is $\int_{\T^d} \rho_0 |v_0|^2 \,dx =0$),
the unique solution provided by Theorems \ref{thm:d=2} and \ref{thm:d=3} is  $v \equiv 0$ for $t>0,$ \emph{even
though $v_0$ need not be $0$}. This is consistent with physics, 
and with the time regularity properties exhibited above.
\end{rmk}

Let us give some insight on the proof of Theorems \ref{thm:d=2} and \ref{thm:d=3}. 
 As in the constant density case,  in order to get  uniqueness, one has to  propagate enough regularity  of the velocity. 
 In the present situation, starting from $H^1$ regularity for the velocity and implementing   a basic energy method on the momentum
 equation,   we will succeed in extracting  some parabolic  smoothing effect \emph{even if the density is rough and vanishes}.
  At the end, we will  have  a control on $v_t,\nabla^2v,\nabla P$ in $L_2(0,T\times\T^d)$  in terms of the initial data. 
 Let us make it more precise~:  after testing  $(INS)_2$ by $v_t,$ it appears that the only troublemaker is the convection term ${\rho} v\cdot\nabla v.$ 
In the case $\inf\rho_0>0,$ the usual approach in dimension $d=2$  (that goes back to \cite{LS}) 
is to combine H\"older inequality  and the following  special case of  Gagliardo-Nirenberg inequalities, first pointed out by O. Ladyzhenskaya in \cite{Lad}:
 \begin{equation}\label{eq:lad}
 \|z\|_4^2\leq C\|z\|_2 \|\nabla z\|_2,
 \end{equation}
to eventually  get for all $\ep>0,$
$$\begin{aligned}
\|{\rho} v\cdot\nabla v\|_2&\leq C\rho^*\|v\|_2^{\frac12}\|\nabla v\|_2\|\nabla^2 v\|_2^{1/2}\\
&\leq \frac C{\varepsilon^{1/3}}(\rho^*)^2\|v\|_2^2\|\nabla v\|_2^2\|\nabla v\|_2^2+\varepsilon \|\nabla^2 v\|_{2}^2.
\end{aligned}$$ 
The last  term may be `absorbed' if $\ep$ is chosen small enough, 
and  the first one may be handled by Gronwall inequality.  {}From it, one gets a global-in-time control 
on $\|v\|_{H^1},$ provided one can bound 
$$\int_0^t\|v\|_2^2\|\nabla v\|_2^2\,dt$$
in terms of the data.   If  \emph{$\rho$ is bounded away from $0$},
then this is a consequence of the basic energy balance  \eqref{eq:uL2}, 
 as   $\|v\|_2\leq (\inf\rho)^{-1/2}\|\sqrt{\rho}\, v\|_2.$
\smallbreak
To handle the case where  we just have $\rho_0\geq0,$  we shall  take advantage of 
the following Desjardins' interpolation inequality (proved in  \cite{Des-CPDE} and in the appendix):  
\begin{equation}\label{eq:deslog}
\biggl(\int \rho|v|^4\,dx\biggr)^{\frac12}\leq C \|\sqrt\rho v\|_{2}\|\nabla v\|_{2}\log^{\frac12}\biggl(e+\frac{\|\rho-M\|^2_{2}}{M^2}+\frac{\rho^*\|\nabla v\|^2_{2}}{\|\sqrt \rho v\|^2_{2}}\biggr)\cdotp
\end{equation}
Note that \eqref{eq:deslog} has just an additional logarithmic term compared to Ladyzhenskaya inequality,
hence using a suitable generalized Gronwall inequality gives us  a chance to get a global control on the solution for all time.  Note also that in (\ref{eq:deslog}) the log correction involves 
just $\|\nabla v\|_2$, not higher norms of $v$.

The three-dimensional  case turns out to be more direct,  if we  assume either smallness 
of $v,$ or  restrict to local-in-time results (the global existence issue for large data being
open, as in the constant density case).

\smallbreak 

Looking back at what we obtained so far, we see that  we  have just $v \in L_2(0,T;H^2),$
hence  we miss   (by a little in dimension $2$ and half a derivative in dimension $3$) the property that 
\begin{equation}\label{eq:Dv}\nabla v\ \hbox{ is in }\  L_1(0,T;L_\infty(\Omega)),\end{equation}
  which, in most fluid mechanics models,   is  (almost) a necessary condition for uniqueness, and   is also strongly 
connected to the existence and uniqueness of a Lipschitz flow for the velocity
(and thus to the possibility of   reformulate System $(INS)$ in Lagrangian coordinates). 
At this stage, the idea is \emph{to shift  integrability from time  to  space variables}, that is
\begin{equation}\label{eq:shift}
 v \in L_2(0,T;H^2) \rightsquigarrow v \in L_{2-\sigma}(0,T;W^2_{2+\delta})\quad\hbox{for suitable } \sigma,\delta>0.
\end{equation}
Indeed, it is clear that if \eqref{eq:shift} holds true  (with  $\delta>1$ in dimension $3$) 
then using  Sobolev embedding gives \eqref{eq:Dv}.    
 Getting \eqref{eq:shift}  will follow from   \emph{time-weighted} estimates, 
a  technique originating  from the theory of parabolic equations that  has  been effectively applied to $(INS)$  recently, in \cite{Li,PZZ}.
In fact, we  prove by means of a standard energy method that 
\begin{equation}\label{eq:vt}
 \sup _{0\leq t \leq T} \biggl(t \int_{\T^d} \rho  |v_t|^2 \,dx\biggr) + \int_{0}^{T} \biggl(t \int_{\T^d} |\nabla v_t|^2 \,dx\biggr)dt \leq C_{0,T}
\end{equation}
with $C_{0,T}$  depending only on $\rho^*,$ $\|\sqrt{\rho_0} v_0\|_2,$ $\|\nabla v_0\|_2$ and $T.$ 
\medbreak
Now, one may  bootstrap the   regularity provided by \eqref{eq:vt}  by rewriting the velocity equation of $(INS)$  in the following elliptic form, treating the time variable
as a parameter: \begin{equation}\label{s1}
 \begin{array}{lcl}
  -\Delta \sqrt t\, v + \nabla \sqrt t\, P =-\rho \sqrt t\, v_t - \sqrt t \,\rho v \cdot \nabla v & \mbox{in} & \Omega,\\[1ex]
\div \sqrt t\, v =0 & \mbox{in} & \Omega.
 \end{array}
\end{equation}
In the 2D case, taking advantage of the classical maximal regularity properties of the Stokes system, as in \cite{DM-arma} (or in 
 \cite{M-QM,MZ-jde} in the context of the compressible Stokes system), 
 we readily get 
 for all $2\leq p\leq\infty$  and   $\varepsilon >0$,
\begin{equation}\label{s2a}
 \nabla^2 \sqrt t\, v, \nabla \sqrt t  P \in    L_p(0,T;L_{p^*-\varepsilon}),\qquad
 p^*:=\frac{2p}{p-2}\cdotp
\end{equation}
Similarly, in the 3D case, we end up with 
\begin{equation}\label{s3a}
 \nabla^2 \sqrt t\, v, \nabla \sqrt t\, P \in    L_p(0,T;L_r)\quad\hbox{with}\quad
 2\leq r\leq \frac{6p}{3p-4}\cdotp
\end{equation}
In both cases, this implies that $\sqrt t \nabla  v$ is in $L_3(0,T;L_\infty),$ and thus, 
by H\"older inequality,
$$
\int_0^T\|\nabla v\|_\infty\,dt=\int_0^T \|\sqrt t\nabla v\|_\infty\,\frac{dt}{\sqrt t}
\leq C T^{1/6} \|\sqrt t\nabla v\|_{L_3(0,T;L_\infty)}.
$$

Finally, having \eqref{eq:Dv} enables   us to reformulate System $(INS)$  in  Lagrangian coordinates (see the beginning of Section \ref{s:uniqueness})
without requiring more regularity on the data than \eqref{eq:data}. This is the key to uniqueness 
 (the direct method  based on stability estimates for  $(INS)$ is bound to fail owing to the hyperbolicity of the mass equation: 
 one lose one derivative, and one cannot afford any loss as $\rho$ is not regular enough). 
As already pointed out in  \cite{DM-cpam,DM-arma}, 
this loss does not occur   if one  looks at the difference between 
 two solutions of $(INS)$ originating from the same initial data, \emph{in Lagrangian coordinates.}
 Estimating that difference may  be done  by means of basic energy arguments.  
The only difficulty is that the divergence is no longer $0$
and one thus first has to solve a `twisted' divergence equation to remove
the non-divergence free part.
Then, ending up with a Gronwall lemma, we get 
uniqueness on a small enough time interval, and arguing by induction yields uniqueness on the existence time interval. 



\section{The proof of existence in Theorems \ref{thm:d=2} and \ref{thm:d=3}}\label{s:existence}

This section mainly concerns the  proof of  a priori estimates for smooth solutions of $(INS).$
Those estimates will eventually enable us to prove  the   existence part of 
 Theorems \ref{thm:d=2} and \ref{thm:d=3} for any data satisfying \eqref{id}.

For notational simplicity, we shall assume throughout that $\mu=1$. This is not restrictive since $(\rho,v,P)(t,\cdot)$ satisfies $(INS)$ with viscosity 
$\mu$ if and only if $(\rho,\mu v,\mu^2P)(\mu t,\cdot)$ satisfies $(INS)$ with viscosity $1.$
Finally, for expository purpose, we shall focus on the torus case. 
As the proof follows from energy arguments, functional embeddings and 
 a Poincar\'e inequality which is  the standard one in the bounded domain case, 
 and is not obvious only in the torus case (see Lemma \ref{l:poincare}),  the case of bounded domain may be treated along the same lines.

\subsection{The persistence of Sobolev regularity}

In the 2D case, the first step  is to prove the following a priori estimate.
\begin{prop}\label{l:H2} Let $(\rho,v)$ be a smooth solution to   System $(INS)$ on $[0,T)\times\T^2,$
satisfying $0\leq\rho\leq\rho^*.$ There exists a constant  $C_0$ depending only on $M,$ $\|\rho_0\|_2,$ $\|\sqrt\rho_0 v_0\|_2$ 
and   $\rho^*$ so that for all $t\in[0,T),$ we have
\begin{equation}\label{b8}
\|\nabla v(t)\|^2_2 + \frac1{2}\int_0^t\Bigl(\|\sqrt\rho v_t\|_2^2 +\!\frac1{\rho^*}\|\nabla^2 v,\nabla P\|_2^2\Bigr)d\tau \leq 
\bigl(e+\|\nabla v_0\|_2^2\bigr)^{\exp\{C_0\|\sqrt\rho_0\,v_0\|_{L_2}^2\}}-e.
\end{equation}
Furthermore, for all $p\in[1,\infty)$ and $t\in[0,T),$ we have
 \begin{equation}\label{eq:vt2}
\|v(t)\|_{p}\leq \frac1M \biggl|\int_{\T^2} (\rho_0 v_0)(x)\,dx\biggr| + C_p\biggl(1+\frac{\|M-\rho_0\|_2}{M}\biggr)\|\nabla v(t)\|_{2}.
\end{equation}
\end{prop}
\begin{proof} It is based on  the following 
 improvement of the Ladyzhenskaya inequality   that has been  pointed out  by B. Desjardins in  \cite{Des-CPDE} 
(see also the  Appendix of the present paper).
\begin{lem}\label{DLest}
 There exists a constant $C$ so that for all $z \in H^1(\T^2)$ and function $\rho\in L_\infty(\T^2)$ 
 with $0\leq\rho\leq\rho^*,$ we have 
\begin{equation}\label{d0}
\biggl(\int \rho z^4\,dx\biggr)^{\frac12}\leq C \|\sqrt\rho z\|_{2}\|\nabla z\|_{2}\log^{\frac12}\biggl(e+\frac{\|\rho-M\|^2_{2}}{M^2}+\frac{\rho^*\|\nabla z\|^2_{2}}{\|\sqrt \rho z\|^2_{2}}\biggr)\cdotp\end{equation}
\end{lem}
Now, testing  the momentum equation of  $(INS)$ by $v_t$ yields:
$$
\begin{aligned}
 \frac 12 \frac{d}{dt} \int_{\T^2} |\nabla v|^2 dx + \int_{\T^2} \rho |v_t|^2 \,dx &=-\int_{\T^2} (\rho v \cdot \nabla v)\cdot v_t\, dx\\ &\leq 
\frac12  \int_{\T^2} \rho |v_t|^2\,dx + \frac12\int_{\T^2} \rho |v\cdot \nabla v |^2\, dx.
\end{aligned}
$$
Hence 
\begin{equation}\label{b1}
\frac d{dt} \int_{\T^2} |\nabla v|^2 dx + \int_{\T^2} \rho |v_t|^2 \,dx
\leq \int_{\T^2} \rho |v\cdot \nabla v |^2\, dx.
\end{equation}
In order to estimate  the second derivatives of $v$ and the gradient of the pressure,   we look at 
$(INS)_2$ in the form
\begin{equation}\label{b3}
 \begin{array}{lcl}
  -\Delta v + \nabla P= - \rho v_t - \rho v \cdot \nabla v & \mbox{in} & (0,T)\times\T^2,\\[1ex]
\div v=0 &\mbox{in} & (0,T)\times\T^2.
 \end{array}
\end{equation}
Then, from the  Helmholtz decomposition on the torus, we get
$$
 \|\nabla^2 v\|_{2}^2+\|\nabla P\|^2_2=\|\rho(v_t+v\cdot\nabla v)\|_2^2
 \leq \frac{\rho^*}2\biggl(\int_{\T^2} \rho |v_t|^2\, dx + \int_{\T^2} \rho |v\cdot\nabla v|^2\, dx\biggr)\cdotp
 $$
 Putting together  with \eqref{b1} thus yields
 \begin{equation}\label{b5}
\frac d{dt} \int_{\T^2} |\nabla v|^2 dx + \frac12\int_{\T^2} \rho |v_t|^2 \,dx
+\frac1{\rho^*}\bigl( \|\nabla^2 v\|_{2}^2+\|\nabla P\|^2_2\bigr)
\leq \frac32\int_{\T^2} \rho |v\cdot \nabla v |^2\, dx.
\end{equation}
In order to bound the last term,  we write that, thanks to \eqref{eq:lad},
\begin{equation}\label{b7}
 \int_{\T^2} \rho |v \cdot\nabla v |^2\, dx \leq \|\sqrt{\rho} |v|^2\|_2 \|\sqrt{\rho} |\nabla v|^2\|_2
 \leq C\sqrt{\rho^*}  \|\sqrt{\rho} |v|^2\|_2 \|\nabla v\|_2\|\nabla^2 v\|_2.
 \end{equation}
 To bound the term $\|\sqrt{\rho} |v|^2\|_2$ despite the fact that $\rho$ vanish, it suffices
 to   combine Inequality \eqref{d0} (after observing that the function $z\mapsto z\log(e+1/z)$ is increasing),  
 the energy balance \eqref{eq:ene} and \eqref{eq:M},  to get
\begin{equation}\label{a4}
  \|\sqrt{\rho}\,|v|^2\|^2_2 \leq C\|\sqrt{\rho_0}\,v_0\|_2^2 \|\nabla v\|_2^2 \log\biggl(e+\frac{\|\rho_0-M\|_2^2}{M^2}+\rho^*\frac{\|\nabla v\|^2_2}{\|\sqrt{\rho_0}\,v_0\|_2^2}\biggr)\cdotp
\end{equation}
Reverting to \eqref{b7}, we end up with 
$$
\begin{aligned}
 \int_{\T^2} \rho |v \cdot\nabla v |^2 dx  &\leq 
 \frac1{3\rho^*}\|\nabla^2 v\|_2^2 +C(\rho^*)^2   \|\sqrt{\rho} |v|^2\|_2^2 \|\nabla v\|_2^2 \\
&\leq\frac1{3\rho^*}\|\nabla^2 v\|_2^2 + C(\rho^*)^{2}
\|\sqrt{\rho_0}\,v_0\|_2^2 \|\nabla v\|_2^4 \log\biggl(e\!+\!\frac{\|\rho_0-M\|_2^2}{M^2}\!+\!\rho^*\frac{\|\nabla v\|^2_2}{\|\sqrt{\rho_0}\,v_0\|_2^2}\biggr)\cdotp
\end{aligned}
$$
Then combining  with \eqref{b5} yields
$$ \frac{d}{dt} \int_{\T^2} |\nabla v|^2 dx + \frac1{2}\int_{\T^2} \Bigl(\rho |v_t|^2 + \frac1{\rho^*}\bigl(|\nabla^2 v|^2+|\nabla P|^2\bigr)\Bigr)dx \leq  C_0 \|\nabla v\|^2_2 \log (e +\|\nabla v\|^2_2) \|\nabla v\|^2_2
$$
with $C_0$ depending only on $\rho^*,$ $\|\sqrt{\rho_0}\,v_0\|_2,$ $M$ and $\|\rho_0-M\|_2.$ 
\medbreak
Denoting $f(t):=C_0\|\nabla v(t)\|_{L_2}^2$ and
$$
X(t):=  \int_{\T^2} |\nabla v(t)|^2\, dx +   \frac1{2}\int_{\T^2} \Bigl(\rho |v_t|^2 + \frac1{\rho^*}\bigl(|\nabla^2 v|^2+|\nabla P|^2\bigr)\Bigr)dx,
$$
the above inequality rewrites 
$$
\frac d{dt} X\leq f X\log(e+X),
$$
from which we get, for all $t\geq0,$
$$ e+X(t)\leq\bigl(e+X(0)\bigr)^{\exp\int_0^tf(\tau)\,d\tau}.$$
Hence, by virtue of \eqref{eq:ene}, we have \eqref{b8}.
\medbreak
In order to prove \eqref{eq:vt2}, we observe that   for all $p \in [1,\infty),$
 denoting by $\bar v(t)$ the average of $v(t)$ on $\T^2,$ we have  by Sobolev embedding,   
\begin{equation}\label{eq:barv}\|v(t)\|_p\leq |\bar v(t)|+\|v(t)-\bar v(t)\|_p\leq  |\bar v(t)|+C_p\|\nabla v(t)\|_2.\end{equation}
Now, from the mass conservation \eqref{eq:M}, we get 
$$M\,\bar v(t) =\int_{\T^2} (\rho v)(t,x)\,dx + \int_{\T^2}(M-\rho(t,x))(v(t,x)-\bar v(t))\,dx.$$
Hence, using \eqref{eq:m},  the conservation of the $L^2$ norm 
of the density, then Cauchy-Schwarz and Poincar\'e inequalities, 
$$ M\,|\bar v(t)| \leq \biggl|\int_{\T^2} (\rho_0 v_0)(x)\,dx\biggr| + \|M-\rho_0\|_{2}\|\nabla v(t)\|_2.$$
Plugging that latter inequality in \eqref{eq:barv} yields \eqref{eq:vt2}.
\end{proof} 


\bigbreak
The adaptation of Proposition \ref{l:H2} to the 3D case reads as follows. 
 \begin{prop}\label{l:H3}
 Let $(\rho,v)$ satisfy System $(INS)$ with $0\leq\rho_0\leq\rho^*$ 
 and divergence free $v_0$ in $H^1(\T^3).$  
 There exists a universal constant $c$ such that if 
 \begin{equation}\label{eq:smalld=3a}
(\rho^*)^{\frac32}\|\sqrt{\rho_0}\, v_0\|_{2}\|\nabla v_0\|_2\leq c
\end{equation}
then for all $t\in[0,T),$ we have 
\begin{equation}\label{eq:H3}
 \|\nabla v(t)\|^2_2 + \frac12\int_0^t\|\sqrt\rho\, v_t\|_2^2 \,d\tau
+\frac1{2\rho^*}\int_0^t\bigl( \|\nabla^2 v\|_{2}^2+\|\nabla P\|^2_2\bigr)\,d\tau\leq 2 \|\nabla v_0\|_{2}^2.
\end{equation}
Furthermore, if \eqref{eq:smalld=3a} is not satisfied then   \eqref{eq:H3}  holds true on $[0,T],$ if
\begin{equation}\label{eq:T}
T\leq\frac{c}{(\rho^*)^6   \|\sqrt\rho_0 v_0\|_2^{2}\|\nabla v_0\|_2^6}\cdotp
\end{equation}
Finally, Inequality \eqref{eq:vt2} holds true for all $p\in[1,6].$
 \end{prop}

\begin{proof} 
Compared to the 2D case,  the only difference is when  bounding the r.h.s. of \eqref{b5}: we write that
 owing to the H\"older inequality and Sobolev embedding $\dot H^1(\T^3)\hookrightarrow L_6(\T^3),$
  $$\begin{aligned} \int_{\T^3}\rho|v\cdot\nabla v|^2\,dx
 &\leq (\rho^*)^{1/2} \|\rho^{1/4} v\|_4^2\|\nabla v\|_4^2\\
 &\leq   (\rho^*)^{3/4}  \|\sqrt{\rho}\,v\|_2^{1/2}\|v\|_6^{3/2}
 \|\nabla v\|_2^{1/2} \|\nabla v\|_6^{3/2}\\
 &\leq   C (\rho^*)^{3/4}  \|\sqrt\rho\, v\|_2^{1/2}\|\nabla v\|_2^{3/2}\|\nabla v\|_2^{1/2} \|\nabla^2v\|_2^{3/2}\\
&\leq\frac1{3\rho^*}\|\nabla^2 v\|_2^2+C(\rho^*)^6   \|\sqrt\rho\, v\|_2^{2}\|\nabla v\|_2^{8}.
 \end{aligned} $$
 Therefore, using \eqref{b5},  we see that 
 $
 \frac d{dt} X\leq f X^3$
 with
 $$\begin{aligned} &X(t):=  \|\nabla v(t)\|^2_2 + \frac12\int_0^t\|\sqrt\rho\, v_t\|_2^2 \,d\tau
+\frac1{2\rho^*}\int_0^t\bigl( \|\nabla^2 v\|_{2}^2+\|\nabla P\|^2_2\bigr)\,d\tau\\
 \hbox{and }\  &f(t):= C(\rho^*)^6   \|\sqrt\rho\, v\|_2^{2}\|\nabla v\|_2^2.\end{aligned}$$
Hence, whenever $T$ satisfies
$ 2X^2(0)\int_0^Tf(t)\,dt<1,$
we have
$$
X^2(t)\leq\frac{X^2(0)}{1-2X^2(0)\int_0^tf(\tau)\,d\tau}\quad\hbox{for all }\ t\in[0,T].
$$
Now, the basic energy conservation \eqref{eq:ene} tells us that
$$
\int_0^Tf(t)\,dt\leq C(\rho^*)^6\|\sqrt{\rho_0}v_0\|_2^4.
$$
Hence, if $(\rho^*)^{\frac32}\|\sqrt{\rho_0}\, v_0\|_{2}\|\nabla v_0\|_2$ is small enough, 
then we have \eqref{eq:H3} for all value of $T.$  
If that  condition is not satisfied, then we observe that 
$$
2X^2(0)\int_0^Tf(t)\,dt\leq\frac12\quad\hbox{implies}\quad
\sup_{t\in[0,T]} X^2(t)\leq 2X^2(0).
$$
Of course,  one can use the fact that
$$
\int_0^T f(t)\,dt \leq C(\rho^*)^6   \|\sqrt\rho_0 v_0\|_2^{2}\:T
\sup_{t\in[0,T]} X(t),
$$
which, together with a bootstrap argument, ensures the second part of the statement.

The proof of the last part of statement goes exactly as in the 2D case. The only difference is that
Sobolev embedding $H^1(\T^3)\hookrightarrow L_p(\T^3)$ holds true only for $p\leq6.$
\end{proof}

 
\subsection{Estimates on time derivatives}

Our next  aim is to  bound  $\sqrt{\rho t}\, v_t$ and $\sqrt t\,\nabla v_t$
in $L_\infty([0,T];L_2)$ and $L_2([0,T];L_2),$   respectively, in terms of the data. 
To achieve it,   let us 
differentiate $(INS)_2$ with respect to $t$:
\begin{equation}\label{t2}
 \rho v_{tt} + \rho_t v_t + \rho_t v \cdot \nabla v + \rho v_t \cdot \nabla v + \rho v \cdot \nabla v_t - \Delta v_t + \nabla P_t =0.
\end{equation}
Then, multiplying \eqref{t2} by $\sqrt t$ yields
$$
 \rho (\sqrt t\, v_t)_t  - \frac1{2\sqrt t} \rho v_t + \sqrt t\,\rho_t v_t + \sqrt t\,\rho_t v \cdot \nabla v + \sqrt t\, \rho v_t \cdot \nabla v + \sqrt t\, \rho v \cdot \nabla v_t 
- \Delta(\sqrt t\, v_t) + \nabla(\sqrt t\, P_t) =0.
$$
Taking the $L^2$ scalar product with  $\sqrt t\, v_t,$ we get
\begin{equation}\label{t4}
 \frac 12 \frac{d}{dt} \int_{\T^d} \rho t |v_t|^2 \,dx + \int_{\T^d} t|\nabla v_t|^2 \,dx = \sum_{i=1}^5 I_i,
\end{equation}
with
\begin{eqnarray}\label{t5}
 I_1&=& \frac12  \int_{\T^d} \rho |v_t|^2 \,dx,\\\label{t6}
 I_2&=& -\int_{\T^d} t \rho_t |v_t|^2 \,dx, \\\label{t7}
 I_3&=&- \int_{\T^d} \bigl(\sqrt t\,\rho_t v \cdot \nabla v\bigr)\cdot( \sqrt t\, v_t) \,dx,\\\label{t8}
 I_4&=& -\int_{\T^d} \bigl(\sqrt t\,\rho v_t \cdot \nabla v\bigr)\cdot(  \sqrt t\, v_t )\,dx,\\\label{t9}
 I_5&=&-\int_{\T^d}\bigl( \sqrt t\, \rho v \cdot \nabla v_t\bigr)\cdot ( \sqrt t\, v_t) \,dx. 
\end{eqnarray}
To bound   $I_2,$ we write that 
$$
\begin{aligned}
 I_2  = \int_{\T^d} t\, \div(\rho v) |v_t|^2\, dx  &\leq 
2\int_{\T^d} t \rho |v|\, |\nabla v_t|\, |v_t|\, dx\\& \leq
C \biggl(\int_{\T^d} \rho t |v_t|^2 \,dx \biggr)^{1/2} 
\biggl(\int_{\T^d} t \rho |v|^2 |\nabla v_t|^2\,dx\biggr)^{1/2} \\&\leq 
C \| \sqrt{\rho t}\, v_t\|_2 \|v\|_\infty \|\sqrt t\, \nabla v_t\|_2\\& \leq 
\frac1{10} \|\sqrt t\, \nabla v_t\|_2^2 + C\|\sqrt{\rho t}\, v_t\|_2^2 \|v\|^2_\infty.
\end{aligned}
$$
The   last term may be controlled in terms of the data thanks to Propositions \ref{l:H2} and \ref{l:H3}. Indeed,
from \eqref{b8} or \eqref{eq:H3}, we  get a bound on $v$ in $L_4(\R_+;L_\infty),$  since 
$$
\|v\|_\infty^4\lesssim \|v\|_2^2\|\nabla^2v\|_{2}^2\quad\hbox{if } d=2;\qquad
\|v\|_\infty^4\lesssim \|\nabla v\|_2^2\|\nabla^2v\|_{2}^2\quad\hbox{if } d=3.
$$
To handle  $I_3,$ we use the continuity equation and perform an integration by parts:
$$
 I_3=-\int_{\T^d} (\sqrt t\, \rho_t v \cdot \nabla v)\cdot (\sqrt t\, v_t)\, dx =
  -\int_{\T^d} t\rho v \cdot \nabla [(v \cdot \nabla v)\cdot v_t]\,dx.
$$
Hence
\begin{equation}\label{t13}
 I_3\leq
\int_{\T^d} t \rho |v| \bigl( |\nabla v|^2 |v_t| + |v|\, |\nabla^2 v|\,| v_t| + |v|\,|\nabla v|\, |\nabla v_t|\bigr)\,dx =: I_{31}+I_{32}+I_{33}.
\end{equation}
To bound $I_{31}$ in the 2D case, we just write that for all $t\in[0,T],$ 
\begin{equation}\label{t14}
 I_{31} =\int_{\T^d}  \sqrt{\rho t}\, |v| |\nabla v| |\nabla v| \,|\sqrt{\rho t}\,v_t|\, dx \leq 
  \|v\|_\infty^2\|\sqrt{\rho t} v_t\|^2_2 + C T \rho^*\|\nabla v\|_4^4,
\end{equation}
and one can use again that $v$ is bounded in $L_2(0,T;L_\infty),$ and that $\nabla v$ is bounded in $L_4(0,T;L_4),$ 
owing to Proposition \ref{l:H2} and Inequality \eqref{eq:lad}. 

This argument fails in the 3D case, but one can combine  H\"older inequality and Sobolev embedding $\dot H^1(\T^3)\hookrightarrow L_6(\T^3)$ to get 
for some constant $C_{T,\rho^*}$ depending only on $T$ and~$\rho^*,$
$$\begin{aligned}
I_{31}&\leq \sqrt{\rho^* T}\|\sqrt{\rho t} v_t\|_4\|v\|_6\|\nabla v\|_{24/7}^2\\
&\leq     \sqrt{\rho^* T} \|\sqrt{\rho t} v_t\|_2^{1/4}\|\sqrt t\, v_t\|_6^{3/4}\|v\|_6\|\nabla v\|_{24/7}^2\\
&\leq\frac1{10}\|\nabla \sqrt t\, v_t\|_2^2+ C_{T,\rho^*}\|\sqrt{\rho t}\, v_t\|_2^{2/5}\|\nabla v\|_{24/7}^{16/5}
\|\nabla v\|_2^{8/5}.
\end{aligned}
$$
Then we observe that 
$$\|\nabla v\|_{24/7}^{16/5}\leq C\|\nabla v\|_2^{6/5}\|\nabla^2 v\|_2^{2},$$
whence
\begin{equation}
 I_{31} \leq \frac1{10}\|\nabla\sqrt t\, v_t\|_2^2+C_{T,\rho^*} \|\sqrt{\rho t} v_t\|_2^{2/5} 
 \|\nabla v\|_2^{14/5}\|\nabla^2 v\|_2^{2}.
\end{equation}
For  $I_{32},$ we have
$$
I_{32}= \int_{\T^d}  t \rho |v|^2 |\nabla^2 v| \,|v_t| \,dx
\leq\rho^* T \|\nabla^2v\|_{2}^2+ \|v\|_\infty^4 \|\sqrt{\rho t}\, v_t\|^2_2,
$$
and one can use that $\nabla^2v\in L_2(\R_+;L_2)$ and that 
$v\in L_4(\R_+;L_\infty),$  as already seen before. 
\medbreak
Finally, for $I_{33},$ we just write that 
$$
 I_{33}=\int_{\T^d} t \rho |v|^2 |\nabla v| \,|\nabla v_t|\,dx \leq C\int_{\T^d} t \rho^2 |v|^4 |\nabla v|^2\,dx+\frac1{10}\int_{\T^d} |\nabla \sqrt t\, v_t|^2 \,dx.
 $$
 The first term of the r.h.s. is under control since 
  $v\in L_4(\R_+;L_\infty)$ and $\nabla v \in L_\infty(\R_+;L_2).$ 
\medbreak
To handle the  term $I_4,$  we write that
$$
\begin{aligned}
 I_4 &\leq \|\nabla v\|_{2}\|\sqrt{\rho t}\, v_t\|_{4}^2\\
 &\leq(\rho^*)^{3/4} \|\nabla v\|_{2}\|\sqrt{\rho t}\, v_t\|_{2}^{\frac12}
 \|\sqrt t\,v_t\|_{6}^{\frac32}
 \\&\leq C(\rho^*)^{3/4} \|\nabla v\|_{2}\|\sqrt{\rho t}\, v_t\|_{2}^{\frac12}
 \|\sqrt t\, \nabla v_t\|_{2}^{\frac32}\\
  &\leq \frac1{10} \|\sqrt t\, \nabla v_t\|_{2}^{2}+C(\rho^*)^3 T
   \|\sqrt\rho v_t\|_{2}^{2}\|\nabla v\|_{2}^4.
\end{aligned}
$$
Finally, for  $I_5,$ we  have to observe that
$$ I_5 = \int_{\T^d} |\rho v|\, |\sqrt t\, \nabla v_t|\, |\sqrt t\, v_t|\,dx\leq \frac1{10}\|\nabla \sqrt t\, v_t\|_{2}^2+C\rho^*\|v\|_\infty^2\|\sqrt{\rho t}\, v_t\|_2^2.
$$
So altogether, if $d=2,$ we get for some constant $C_{T,\rho^*}$ depending only on $\rho^*$ and $T,$   
$$
\displaylines{
  \frac{d}{dt}\|\sqrt{\rho t}\, v_t\|_2^2 + \|\nabla \sqrt t\, v_t\|_2^2 
  \leq   C\bigl((1+\rho^*)\|v\|_\infty^2+\|v\|_\infty^4\bigr) \|\sqrt{\rho t} \,v_t\|_2^2
 \hfill\cr\hfill+C_{T,\rho^*}\bigl(\|\nabla v\|_4^4+\|\nabla^2v\|_2^2+\|v\|_\infty^4 \|\nabla v\|_2^2
 +\|\sqrt\rho v_t\|_2^2(1+\|\nabla v\|_2^4)\bigr),}
$$
and if $d=3,$
$$
\displaylines{
  \frac{d}{dt} \|\sqrt{\rho t}\,  v_t\|_2+ \|\nabla \sqrt t\,v_t\|_2^2 
  \leq   C \bigl((1+\rho^*)\|v\|_\infty^2+\|v\|_\infty^4\bigr) \|\sqrt{\rho t} \,v_t\|_2^2
 \hfill\cr\hfill+C_{T,\rho^*}\bigl(\|\sqrt{\rho t}\, v_t\|_2^{2/5}
  \|\nabla v\|_2^{14/5}\|\nabla^2v\|_2^2  +\|\nabla^2v\|_2^2+\|v\|_\infty^4 \|\nabla v\|_2^2
 +\|\sqrt\rho v_t\|_2^2(1+\|\nabla v\|_2^4)\bigr).}
$$
The above two inequalities rewrite
\begin{equation}\label{t15}
  \frac{d}{dt} \biggl( \|\sqrt{\rho t}\, v_t\|_2^2+ \int_0^t\tau \|\nabla v_t\|_2^2 \,d\tau\biggr)
  \leq   h(t)\bigl(1+ \|\sqrt{\rho t}\,  v_t\|_2^2\bigr)\end{equation}
with $h\in L_{1,loc}(\R_+)$ depending only  on  $\rho^*,$ $\|\sqrt \rho_0\, v_0\|_2$ and  $\|\nabla v_0\|_2.$ 
\medbreak
Obviously,  if the solution is smooth with density bounded away from zero, then 
we have
$$\lim_{t\to 0^+} \int_{\T^d}\rho t|v_t(t,x)|^2\,dx=0.
$$
Therefore, an obvious time integration in \eqref{t15} yields
\begin{equation}\label{eq:h}
 \|\sqrt{\rho t}\, v_t\|_2+ \int_0^t\|\nabla\sqrt \tau \, v_t\|_2^2 \,d\tau \leq \exp\biggl\{\int_0^t h(\tau)\,d\tau\biggr\}-1.
\end{equation}
Using  the same argument starting
from time $t_0,$   one arrives at the following lemma:
\begin{lem}\label{l:t-est}
 Assume $d=2,3$  and that the solution is smooth with no vacuum. Then  for  all $t_0,T\geq0,$  we have
\begin{equation}\label{eq:t-est}
 \sup _{t_0\leq t \leq t_0+T} \int_{\T^d} \rho (t-t_0) |v_t|^2 \,dx + \int_{t_0}^{t_0+T}\int_{\T^d}(t-t_0) |\nabla v_t|^2 \,dx\,dt  \\ 
 \leq c(T)  
\end{equation}
with $c(T)$ going to zero as $T \to 0.$ 
\end{lem}
From \eqref{eq:h}, one can deduce that  for all $p<\infty$
if $d=2$ (for all $p\leq 6$ if $d=3$), we have
$$\|\sqrt t v_t\|_{L_2(0,T;L_p)}\leq c(T)\quad\hbox{with } c(T)\to 0\ \hbox{for } T\to0.$$
Indeed, denoting by   $ \overline{(v_t)}$  the average of $v_t,$ one can write that
$$
\int_{\T^d} \rho v_t\,dx=M \overline{(v_t)}+\int_{\T^d} \rho(v_t- \overline{(v_t)})\,dx.
$$
Hence,
$$ M |\overline{ (v_t)}|\leq\|\rho\|_2\|\nabla v_t\|_2 + M^{1/2} \|\sqrt\rho\, v_t\|_2.$$
Consequently, by Sobolev embedding,  and because $\|\rho\|_2$ and $M$ are  time independent, 
$$
\|v_t\|_{p}\leq\|v_t- \overline{(v_t)}\|_{p}+|\overline{(v_t)}|
\leq\biggl(C_p+\frac{\|\rho_0\|_2}{M}\biggr)\|\nabla v_t\|_{2}+\frac1{M^{1/2}}\|\sqrt \rho\, v_t\|_2.
$$
This implies that, for all $p<\infty$ if $d=2$ and for all $p\leq6$ if $d=3,$ we have
\begin{equation}\label{t18}
 \|\sqrt t\, v_t\|_{L_2(0,T;L_p)} \leq 
 \biggl(C_p+\frac{\|\rho_0\|_2}{M}\biggr)
 \|\sqrt t\, \nabla v_t\|_{L_2(0,T;L_2)} +\frac1{M^{1/2}}\|\sqrt{\rho t}\, v_t\|_{L_2(0,T;L_2)}. 
\end{equation}

Another consequence of \eqref{eq:h} is that we have some  control on  the regularity of $v$ with respect to 
the time variable. This  is given by the following lemma. 
\begin{lem}  Let  $p\in[1,\infty]$  and $z:(0,T)\times\T^d\to\R$ satisfy 
$z\in L_2(0,T;L_p)$ and $\sqrt t\, z_t\in   L_2(0,T;L_p).$
Then $z$ is in $H^{\frac12-\alpha}(0,T;L_p)$ for all $\alpha\in(0,1/2)$
and  we have
\begin{equation}\label{eq:fractionnaire}
\| z\|_{H^{\frac12-\alpha}(0,T;L_p)}^2\leq \|z\|_{L_2(0,T;L_p)}^2 +
C_{\alpha,T}\|\sqrt t\, z_t\|_{L_2(0,T;L_p)}^2,
\end{equation}
with $C_{\alpha,T}$ depending only on $\alpha$ and on $T.$
\end{lem}
\begin{proof} The proof  relies on the definition of Sobolev norms in terms of finite differences.
Indeed, we have
$$
\| z\|_{H^{\frac12-\alpha}(0,T;L_p)}^2= \|z\|_{L_2(0,T;L_p)}^2 +\int_0^T\biggl(\int_0^{T-h}\frac{\|v(t+h)-v(t)\|_p^2}{h^{2-2\alpha}}\,dt\biggr)dh.
$$
Now, we observe that
$$\begin{aligned}
\int_0^{T-h}\|v(t+h)-v(t)\|_p^2\,dt&\leq
\int_0^{T-h} \Big\| \int_t^{t+h} \sqrt s v_t (s)\, \frac{ds}{\sqrt s} \Big\|_{p}^2\,dt\\
&\leq\int_0^{T-h} \left(\int_t^{t+h} s^{-1}\, ds\right) \left(\int_t^{t+h} s \|v_t (s)\|_{p}^2 \,ds\right) dt\\
&\leq \|\sqrt t\, v_t\|_{L_2(0,T;L_p)}^2\int_0^{T-h} \biggl(\int_t^{t+h} s^{-1}\, ds\biggr)dt.
\end{aligned}$$
{}From Fubini theorem, it is not difficult to see that, if $0<\alpha<1/2,$
$$
\int_0^Th^{2\alpha-2}\biggl(\int_0^{T-h} \biggl(\int_t^{t+h}\, \frac{ds}s\biggr)dt\biggr)dh
=\frac{1}{1-2\alpha}\int_0^T\biggl(\int_t^T\bigl((s-t)^{2\alpha-1}-T^{2\alpha-1}\bigr)\,\frac{ds}s\biggr)dt.
$$
Therefore, 
$$
\int_0^T\biggl(\int_0^{T-h}\frac{\|v(t+h)-v(t)\|_p^2}{h^{2-2\alpha}}\,dt\biggr)dh\leq C_{\alpha,T}  \|\sqrt t\, v_t\|_{L_2(0,T;L_p)}^2,
$$
which completes the proof of the lemma. 
\end{proof}


\subsection{Shift of integrability} 

The results that we proved so far will enable us   to bound  $\nabla v$ in $L_1(0,T;L_\infty),$  in terms of the data and of $T.$  
This will be achieved thanks to the `shift  of integrability' method alluded to in the introduction.
 Let us first examine  the 2D case.
 \begin{lem}\label{l:high} If $d=2$ then for all $T>0,$ $p\in[2,\infty]$ and $\varepsilon$ small enough, we have
\begin{equation}\label{s7a}
\| \nabla^2 \sqrt t\, v\|_{L_p(0,T;L_{p^*-\varepsilon})}
+\|\nabla \sqrt t\, P\|_{L_p(0,T;L_{p^*-\varepsilon})}\leq C_{0,T}
\end{equation}
where  $p^*:=\frac{2p}{p-2},$ and $C_{0,T}$ depends only on $\rho^*$, $\|\sqrt{\rho_0}\,v_0\|_{2},$ 
$\|\nabla v_0\|_{2},$ 
 $p$ and $\ep.$
\medbreak
Furthermore, 
for all $1\leq s<2,$ there exists $\beta>0$ such that  
\begin{equation}\label{s8}
\int_0^T \|\nabla v(t)\|^s_\infty\,dt \leq C_{0,T}\, T^\beta.
\end{equation}
\end{lem}
\begin{proof} 
{}From  $(INS)_2,$ we gather that 
\begin{equation}\label{s1a}
 \begin{array}{lcr}
  -\Delta \sqrt t\, v + \nabla \sqrt t\, P=-\sqrt{\rho t}\, v_t - \sqrt t\, \rho v \cdot \nabla v & \mbox{ in } & (0,T)\times\T^2,\\[1ex]
\div \sqrt t\, v =0 & \mbox{ in } & (0,T)\times\T^2.
 \end{array}
\end{equation}
As $\rho$ is bounded and $\sqrt{\rho t}\, v_t$ is in $L_\infty(0,T;L_2),$  we have  
$\rho \sqrt t\, v_t$ in $L_\infty(0,T;L_2),$ too.  Furthermore, according to \eqref{t18},
$\sqrt t\, v_t$ (and thus $\rho \sqrt t\, v_t$) is in $L_2(0,T;L_q)$ for all finite $q.$
Therefore,  by H\"older inequality, we  have
\begin{equation}\label{s2}
 \|\rho \sqrt t\,v_t \|_{L_p(0,T;L_r)}\leq C_{0,T}\quad\hbox{for all }\ p\in[2,\infty]\ \hbox{ and }\ 
 2\leq r< p^*.
\end{equation}
Similarly, the bounds for $\nabla v$ in $L_\infty(0,T;L_2)\cap L_2(0,T;H^1)$ imply that 
\begin{equation}\label{s3}
\|\nabla v\|_{L_p(0,T;L_r)}\leq C_{0,T} \quad\hbox{for all }\  p\geq2\ \hbox{ and }\ r<p^*,
\end{equation}  and,  as obviously, 
$v$ (and thus $\sqrt t\,\rho v$) is bounded in all spaces $L_q(0,T;L_r)$ (except $q=r=\infty$), one can
conclude that 
 \begin{equation}\label{s4}
 \|\sqrt t\, \rho v \cdot \nabla v \|_{L_p(0,T;L_{r})}\leq C_{0,T}\quad\hbox{for all }\ p\in[2,\infty]\ \hbox{ and } 2\leq r<p^*. 
  \end{equation}
Then, the maximal regularity estimate for the Stokes system implies that
\begin{equation}\label{s7}
 \| \nabla^2 \sqrt t\, v, \nabla \sqrt t\, P\|_ {L_p(0,T;L_r)}\leq C_{0,T}\ \hbox{ for all }\
 2\leq p\leq\infty\ \hbox{ and }\ 2\leq r<p^*. 
\end{equation}

To prove \eqref{s8}, fix $p\in[2,\infty)$ so that $ps<2(p-s)$ and $1\leq s<2,$ then    take $r\in]2,p_*[$  (so that we have the embedding 
$W^{1}_{r}\hookrightarrow L_\infty$). 
We get, remembering \eqref{s3}, 
$$\begin{aligned}
 \biggl(\int_0^T\|\nabla v\|^s_\infty\,dt\biggr)^{\frac1s} &\lesssim
 \biggl(\int_0^T \bigl(t^{-1/2} \,\|\sqrt t\, \nabla v\|_{W^1_{r}}\bigr)^s\, dt\biggr)^{\frac1s}\\
  &\lesssim\biggl(\int_0^T t^{-\frac{ps}{2p-2s}}\,dt\biggr)^{\frac1s-\frac1p}
   \|\nabla\sqrt t\, v\|_{L_{p}(0,T;W^1_{r})} \leq C_{0,T}\,T^{\frac{2p-2s-ps}{2ps}},\end{aligned}
$$ whence the desired result.
\end{proof}

In  the 3D case,  Lemma \ref{l:high} becomes:
 \begin{lem}\label{l:high3}
 For all $T>0$ and $p\in[2,\infty],$  we have
\begin{equation}\label{s7b}
\| \nabla^2 \sqrt t\, v\|_{L_p(0,T;L_r)}+\|\nabla\sqrt t\, P\|_{L_p(0,T;L_r)}\leq C_{0,T}    \quad\hbox{for}\quad
 2\leq r\leq \frac{6p}{3p-4},
\end{equation}
 where $C_{0,T}$ depends only on $\rho^*$, $\|\sqrt{\rho_0}\,v_0\|_{2},$ $\|\nabla v_0\|_{2},$  $p$ and $\ep.$
\smallbreak
Furthermore, if   $1\leq s<4/3,$   then we have for some $\beta>0,$
\begin{equation}\label{s8-3}
\int_0^T \|\nabla v(t)\|^s_\infty\,dt \leq  C_{0,T}\,T^\beta.
\end{equation}
\end{lem}
\begin{proof} 
From Lemma \ref{l:t-est} and the embedding of $\dot H^1(\T^3)$ in $L_6(\T^3),$ 
we readily get that
\begin{equation}\label{s2-3}
 \|\rho \sqrt t\,v_t \|_{L_p(0,T;L_r)}\leq C_{0,T}\quad\hbox{for all }\ p\in[2,\infty]\ \hbox{ and }\ 
 2\leq r\leq \frac{6p}{3p-4}\cdotp
\end{equation}
We claim that we have the same type of bound for $\sqrt t\,\rho v\cdot\nabla v.$
   However, one has to proceed in two steps to get the full range of indices. 
   As a start, we observe that (both properties
   being just consequences of Theorem \ref{l:H3} and obvious embedding):
   $$
   \sqrt t\,\rho v\in L_4(0,T;L_\infty)\quad\hbox{and}\quad
   \nabla v\in L_4(0,T;L_3)\quad\hbox{implies }\ \sqrt t\,\rho v\cdot\nabla v\in L_2(0,T;L_3),
   $$
   and that, similarly
   $$\sqrt t\,\rho v\in L_\infty(0,T;L_6)\quad\hbox{and}\quad
   \nabla v\in L_\infty(0,T;L_2)\quad\hbox{implies }\ \sqrt t\,\rho v\cdot\nabla v\in L_\infty(0,T;L_{3/2}).
   $$
Therefore, interpolating and using H\"older inequality yields 
$$
 \sqrt t\,\rho v\cdot\nabla v\in L_p(0,T;L_r)\quad\hbox{for all}\quad
 (p,r) \ \hbox{ such that }\ p\geq2\ \hbox{ and }\ \frac2p+\frac3r=2.
 $$
 Then using the maximal regularity properties of \eqref{s1a} yields
 \begin{equation}\label{s4-3}
\|\nabla^2\sqrt t\,v,\ \nabla \sqrt t\,P\|_{L_p(0,T;L_r)}\leq C_{0,T}\quad\hbox{for all}\quad
 (p,r) \ \hbox{ such that }\ p\geq2\ \hbox{ and }\ \frac2p+\frac3r=2.
 \end{equation}
 From this, using  the bound for  $\rho v$ in $L_\infty(0,T;L_6)$ 
 and  the embedding $W^1_r(\T^3)\hookrightarrow L_q(\T^3)$ with $\frac3q=\frac3r-1$ if $1\leq r<3$
 (which implies that $\nabla\sqrt t\, v$ is bounded in $L_p(0,T;L_q)$ with  $\frac2p+\frac3r=2$),
  one gets  \eqref{s4-3} for the full range of indices. 
 \smallbreak
 To prove \eqref{s8-3}, fix $p\in(2,4)$ such that  $ ps <2p-2s$ (our condition on $s$ makes it possible),  
 and take   $r=\frac{6p}{3p-4}\cdotp$  Using that 
$W^{1}_{r}\hookrightarrow L_\infty$ (because  $r>3$ for $2<p<4$), one can write that
$$
 \biggl(\int_0^T\|\nabla v\|^s_\infty\,dt\biggr)^{\frac1s} \lesssim
 \biggl(\int_0^T \|\sqrt t\, \nabla v\|_{W^1_{r}}^s\, \frac{dt}{\sqrt t}\biggr)^{\frac1s}
  \lesssim\biggl(\int_0^T t^{-\frac{ps}{2p-2s}}\,dt\biggr)^{\frac1s-\frac1p}
   \|\nabla\sqrt t\, v\|_{L_{p}(0,T;W^1_{r})} 
$$ whence  the desired result.
\end{proof}


\subsection{The proof of existence}

Here we  briefly explain the issue of existence.  For expository purpose, we focus on global results (that is either $d=2$ 
or $d=3$ and the velocity is small), and leave to the reader the  construction of local solutions in the case of large $v_0,$  if $d=3.$

The general idea is to take advantage of classical results to construct smooth solutions corresponding to smoothed out approximate data with no vacuum, 
then to pass to the limit. 
More precisely, consider
$$
 v^\ep_0 \in C^{\infty}(\T^d)\ \hbox{ with }\ \div v^\varepsilon_0=0, \ \ \mbox{ and }\  \rho^\varepsilon_0 \in C^{\infty}(\T^d) \ \hbox{ with }\   \ep\leq \rho^{\varepsilon}_0\leq\rho^*,
$$
such that
\begin{equation}
 v^\varepsilon_0 \to v_0 \mbox{ in } H^1, \;\;\;\;\;  \rho^\varepsilon_0 \rightharpoonup \rho_0 \mbox{ in } L_\infty\ \hbox{weak *}
\  \hbox{ and }   \rho_0^\ep\to \rho_0\mbox{ in } L_p, \ \hbox{if } p<\infty.
\end{equation}

Then, in light of the classical strong solution theory for  $(INS)$  (see \cite{LS} and more recent developments in  \cite{Dan1}), 
there exists a unique global smooth solution $(\rho^\ep,v^\ep,P^\ep)$ corresponding to data $(\rho_0^\ep,v_0^\ep),$ and  satisfying $\ep\leq \rho^\ep\leq\rho^*.$
\smallbreak
Being smooth, the triplet  $(\rho^\ep,v^\ep,P^\ep)$ satisfies  all the priori estimates of the previous subsection, and thus in particular,
\begin{equation}\label{E3}
\|\sqrt{\rho^\ep}v^\ep_t,\nabla P^\ep\|_{L_2(\R_+\times\T^d)}+
  \|v^\varepsilon\|_{L_\infty(\R_+;H^1) \cap L_2(\R_+;H^2)} \leq C(\|\sqrt{\rho} v\|_2,\|\nabla v_0\|_2,\rho^*),
  \end{equation}
  and also, thanks to \eqref{eq:h},
    \begin{multline}\label{E4}  \sup_{t\in [t_0,t_0+T)}\Bigl((t-t_0) \|\sqrt{\rho^\ep(t)} v^\ep_t(t)\|_2^2\Bigr) \\+ 
\int_{t_0}^{t_0+T}(t-t_0) \|\nabla v_t^\ep\|_2^2\, dt \leq C(T,\|\sqrt{\rho_0} v_0\|_2,\|\nabla v_0\|_2,\rho^*),
\end{multline}
and, according to Inequality \eqref{eq:fractionnaire}, for $\alpha\in(0,1/2),$ 
 $$  \|v^\varepsilon\|_{H^{\frac12-\alpha}(0,T;L_2)} \leq C(T,\|\sqrt{\rho_0} v_0\|_2,\|\nabla v_0\|_2,\rho^*).$$
 Interpolating with \eqref{E3} yields for small enough $\eta>0,$
 \begin{equation}\label{E5}
   \|v^\varepsilon\|_{H^{\eta}(0,T\times \T^d)} \leq C(T,\|\sqrt{\rho_0} v_0\|_2,\|\nabla v_0\|_2,\rho^*).
   \end{equation}
 
 The bounds on $v^\ep$ and standard compact embedding imply  that, up to subsequence,  $v^\varepsilon \to v \mbox{ in } L_{2,loc}(\R_+;H^1)$
 for some $v$ that, in addition, satisfies \eqref{E3}  (as regards $\nabla P^\ep$ and $v^\ep$) and \eqref{E5}. 
 For the density, we have  $ \rho^\varepsilon \rightharpoonup \rho \mbox{ in } L_{\infty}(\R_+;L_\infty)$
 and $0\leq\rho\leq\rho^*.$  All those informations are more than enough to justify that  $(\rho,v)$  is a weak solution to $(INS),$ namely  
\begin{equation}\label{E7}
(\rho(t) v(t),\phi(t)) -(\rho_0 v_0,\phi_0) -\int_0^t (\rho v, \phi_t)\,d\tau 
-\int_0^t (\rho v \otimes v,\nabla \phi)\,d\tau + \int_0^t(\nabla v,\nabla \phi)\,d\tau=0 
\end{equation}
for all smooth compactly supported divergence-free vector function $\phi \in C^{\infty}([0,\infty) \times \T^d; \R^d),$
and that  the continuity equation is fulfilled in a distributional meaning:
\begin{equation}\label{E9}
 \rho_t + \div (\rho v) =0 \mbox{ in } \mathcal{D}'([0,\infty)\times\T^d).
\end{equation}
 Now,  arguing as in e.g. \cite{DZP}, page 2405, one can show that
  $ (\rho^\varepsilon)^2 \rightharpoonup(\rho)^2 \mbox{ in } L_{\infty}(\R_+;L_\infty),$
  which eventually implies that 
  \begin{equation}\label{eq:rho}
  \rho^\varepsilon \to \rho\ \hbox{  in }\  \cC(\R_+;L_p)\ \hbox{ for all finite }\ p.\end{equation}
  Therefore  \eqref{E4} is satisfied by $(\rho,v).$
Furthermore,   \eqref{E3} and \eqref{E9} applied to the formulation \eqref{E7} yields that the momentum equation is fulfilled in the strong sense, i.e.
\begin{equation}
 \rho v_t + \rho v \cdot \nabla v - \Delta v +\nabla P = 0 \mbox{ in } L_2(\R_+;L_2)
\end{equation}
for some pressure function $\nabla P$ in $L_2(\R_+;L_2)$ that satisfies \eqref{E3}. 
\medbreak
Of course, being smooth, the solution $(\rho^\ep,v^\ep,P^\ep)$ fulfills \eqref{eq:ene}, \eqref{eq:m}, \eqref{eq:M}, \eqref{eq:infsup}, and thus
\begin{equation}\label{eq:eneep}
\|\sqrt{\rho^\ep(t)}\,v^\ep(t)\|_2^2+2\int_{t_0}^t\|\nabla v^\ep\|_2^2\,d\tau= \|\sqrt{\rho^\ep(t_0)}\,v^\ep(t_0)\|_2^2\quad\hbox{for all }\ t,t_0\geq0.
\end{equation}
By construction, in the case $t_0=0,$ the last term tends to $\|\sqrt\rho_0 v_0\|_2^2,$ and the fact that $v^\ep\to v$ in $L_{2,loc}(\R_+;H^1)$ guarantees
that, for all $t\geq0,$  the second term converges to $\int_{t_0}^t\|\nabla v\|_2^2\,d\tau.$
Next, \eqref{eq:h} guarantees that $\nabla v_t^\ep$ is bounded in $L_2(t_0,T\times\T^d)$ 
for all $0<t_0<T.$  Then combining with \eqref{t18}  gives boundedness of
$v^\ep_t$ in $L_2(t_0,T;H^1).$  Now, because $v^\ep$ is bounded in $L_\infty(t_0,T;H^1)$
thanks to Prop. \ref{l:H2} and \ref{l:H3}, one can conclude by means of Ascoli theorem 
that, up to extraction, $v^\ep\to v$ strongly in $\cC([t_0,T];L_p)$ for all $p<6$ (and even $p<\infty$ if $d=2$).
Combining with \eqref{eq:rho} ensures that $\sqrt\rho v$ is in $\cC(]0,+\infty[;L_2)$ and
that one can pass to the limit in all the terms of \eqref{eq:eneep}: we eventually get 
\begin{equation}\label{eq:enebis}\|\sqrt{\rho(t)}\,v(t)\|_2^2+2\int_{t_0}^t\|\nabla v\|_2^2\,d\tau= \|\sqrt{\rho(t_0)}\,v(t_0)\|_2^2\quad\hbox{for all }\ t,t_0\geq0.\end{equation}
Finally, to get the  strong continuity of  $\sqrt \rho v$ at $t=0,$ we notice that the uniform bounds ensure
the weak continuity, and that \eqref{eq:enebis} gives
 $$
\|\sqrt{\rho(t)}\,v(t)\|_2^2\longrightarrow \|\sqrt{\rho_0}\,v_0\|_2^2\quad\hbox{for }\ t\to0.
$$
This completes the proof of the existence part of Theorems \ref{thm:d=2} and \ref{thm:d=3}.\qed

\section{The proof of uniqueness}\label{s:uniqueness}

As there is no hope to prove  uniqueness of solutions at the level of the  Eulerian coordinates system, owing to the lack 
of regularity of the density,  we shall prove it for the solutions written in the Lagrangian 
coordinates.
\medbreak
To this end, we introduce the flow
$X:\R_+\times\T^d\to\T^d$  of $v,$ that is   the unique solution to \eqref{l1}. 
In Lagrangian coordinates $(t,y)$, a solution $(\rho,v,P)$ to $(INS)$ recasts in $(\eta,u,Q)$ with
\begin{equation}\label{l2}
 \eta(t,y):=\rho(t,X(t,y)), \qquad u(t,y):=v(t,X(t,y))\quad\hbox{and}\quad Q(t,y):=P(t,X(t,y)).
\end{equation}
Note that
$$
X(t,y)=y+\int_0^t v(\tau,X(\tau,y))\,d\tau=\int_0^t u(\tau,y)\,d\tau,
$$
 and thus 
 $$
 \nabla_yX(t,y)=\Id+\int_0^t \nabla_yu(\tau,y)\,d\tau.
 $$
 Let $A(t):=(\nabla_yX(t,\cdot))^{-1}.$   In the $(t,y)$-coordinates, operators $\nabla,$  $\div$
 and $\Delta$  translate into 
\begin{equation}\label{l4}
 \nabla_{u}:={}^T\!A\, \nabla_y, \quad \divu := {}^T\!A:\nabla_y =\divy(A \cdot)\
 \hbox{ and }\ \Delta_u:=\divy(A\,{}^T\!A\nabla_y\cdot),
\end{equation}
and  the triplet  $(\eta,u,Q)$ thus satisfies
\begin{equation}\label{LNS}
 \begin{array}{lcr}
  \displaystyle \eta_t=0 & \mbox{in} & (0,T)\times \T^d,\\
\displaystyle \eta u_t - \Delta_u u +\nabla_u Q =0 & \mbox{in} & (0,T)\times \T^d,\\
\displaystyle  \divu u =0 & \mbox{in} & (0,T)\times \T^d.
 \end{array}
\end{equation}
As pointed out in e.g.  \cite{DM-cpam,DM-arma},
in our regularity framework, that latter system is equivalent to $(INS)$ whenever, say, 
 \begin{equation}\label{eq:smallDu}
 \int_0^T\|\nabla u\|_{\infty}\,d\tau\leq\frac12\cdotp
 \end{equation}
Of course, if that condition is fulfilled then one may write that 
 \begin{equation}\label{eq:A}
 A=(\Id+(\nabla_yX-\Id))^{-1}=\sum_{k=0}^{+\infty}(-1)^k\biggl(\int_0^t \nabla_yu(\tau,\cdot)\,d\tau\biggr)^k.
 \end{equation}
Let us tackle  the proof of uniqueness in  the torus case (the bounded domain case goes exactly the same : just change Lemma \ref{l:div} to  Lemma \ref{l:div-dom}). 
\medbreak
Consider two solutions $(\rho^1,v^1,P^1)$ and $(\rho^2,v^2,P^2)$ of $(INS)$, 
emanating from  the same initial data, and fulfilling the properties of Theorems \ref{thm:d=2} and \ref{thm:d=3},
 and denote by $(\eta^1,u^1,Q^1)$ and $(\eta^2,u^2,Q^2)$ the corresponding  triplets in Lagrangian coordinates. 
 Of course,  we have $\eta^1=\eta^2=\rho_0$, which explains the choice of our approach here. 
  In what follows, we shall use repeatedly the fact that for $i=1,2,$ we have (recall \eqref{t18})
  \begin{multline}\label{eq:regularity}
 \qquad t^{1/2}\nabla u^i\in L_2(0,T;L_\infty),\quad
  \nabla u^i\in L_1(0,T;L_\infty)\cap L_4(0,T;L_3)\cap L_2(0,T;L_6),\\
  u^i\in L_4(0,T;L_\infty),\quad t^{1/2}\nabla Q^i\in L_2(0,T;L_4)\ \hbox{ and }\ 
 t^{1/2}u_t^i\in L_{4/3}(0,T;L_6).\qquad
 \end{multline}
Denoting  $\du:=u^2-u^1$ and $\dQ:=Q^2-Q^1,$  we get 
\begin{equation}\label{NSL-uniq}
 \begin{aligned}
  &{\rho_0}\du_t - \Delta_{u^1}\du + \nabla_{u^1}\dQ= 
(\Delta_{u^2} -\Delta_{u^1}) u^2 - (\nabla_{u^2}-\nabla_{u^1}) Q^2,\\
&\div_{\!u^1}\du=(\div_{\!u^1}-\div_{\!u^2})u^2,\\
&\du|_{t=0}=0.
\end{aligned}
\end{equation}
We claim that  for sufficiently small $T>0,$  we have 
$$
\int_0^T\int_{\T^d}\bigl(|\du(t,y)|^2+|\nabla\du(t,y)|^2\bigr)dy\,dt=0.
$$
To prove our claim,  decompose $\du$ into 
\begin{equation}\label{l5}
 \du = w+z,
\end{equation}
where $w$ is  the solution given by Lemma \ref{l:div} to the following problem:
\begin{equation}\label{l6}
 \div_{\!u^1} w = (\div_{\!u^1}-\div_{\!u^2})u^2=\div (\dA \, u^2)= {}^T\!\dA: \nabla u^2,
\end{equation}
with $\dA:=A^2-A^1$ and $A^i:=A(u^i).$
\medbreak
As a first step in the proof of our claim, let us establish the following lemma:
 \begin{lem}
The solution $w$ to \eqref{l6} given by Lemma \ref{l:div} satisfies 
\begin{equation}\label{eq:w}
\|w\|_{L_4(0,T;L_2)}+ \|\nabla w\|_{L_2(0,T\times\T^d)} + \|w_t\|_{L_{4/3}(0,T;L_{3/2})} \leq c(T)\|\nabla\du\|_{L_2(0,T\times\T^d)}
\end{equation}
with $c(T)$ going to $0$ when $T$ tends to $0.$
 \end{lem}
\begin{proof} 
Lemma \ref{l:div} and Identity \eqref{eq:A} ensure that  there exist two universal positive constants $c$ and $C$ such that if  
\begin{equation}\label{eq:smallu}
\|\nabla u^1\|_{L_1(0,T;L_\infty)}+\|\nabla u^1\|_{L_2(0,T;L_6)}\leq c, 
\end{equation}
then the following inequalities hold true:
\begin{equation}\label{eq:www}\begin{array}{c}
\|w\|_{L_4(0,T;L_2)}\leq C\|\dA \,u^2\|_{L_4(0,T;L_2)}, \quad
\|\nabla w\|_{L_2(0,T;L_2)}\leq C\|{}^T\!\dA:\nabla u^2\|_{L_2(0,T;L_2)}\\[1ex] \hbox{\rm and }\ 
\|w_t\|_{L_{4/3}(0,T;L_{3/2})}\leq C\|\dA\, u^2\|_{L_{4}(0,T;L_{2})}+C\|(\dA \,u^2)_t\|_{L_{4/3}(0,T;L_{3/2})}.
\end{array}\end{equation}
In all that follows, $c(T)$ designates a nonnegative continuous increasing function of $T,$ with $c(0)=0.$
 Now, let us bound the r.h.s. of \eqref{eq:www}. 
 Regarding  ${}^T\!\dA:\nabla u^2,$  one can  use the fact that
 if both $u^1$ and $u^2$ fulfill \eqref{eq:smallu}, then we have
\begin{equation}\label{l128}
 \sup_{t\in[0,T]}  \|t^{-1/2} \dA\|_2 \leq C\sup_{t\in[0,T]}\Big\|t^{-1/2}\int_0^t \nabla \du\,d\tau\Bigr\|_{2}
  \leq C\|\nabla\du\|_{L_2(0,T;L_2)}.
\end{equation}
This  stems from H\"older inequality and the following identity: 
\begin{equation}\label{eq:dA}
\dA(t)=\biggl(\int_0^t \nabla \du\,d\tau\biggr)\cdot
\biggl(\sum_{k\geq1}\sum_{0\leq j<k} C_1^jC_2^{k-1-j}\biggr)
\quad\hbox{with}\quad
C_i(t):=\int_0^t \nabla u^i\,d\tau.
\end{equation}
Therefore, thanks to \eqref{eq:regularity} and \eqref{l128}, we have
$$\begin{aligned}
\|{}^T\!\dA:\nabla u^2\|_{L_2(0,T\times\T^d)}&\leq 
 \sup_{t\in[0,T]}  \|t^{-1/2} \dA(t)\|_2\|t^{1/2}\nabla u^2\|_{L_2(0,T;L_\infty)}\\ &\leq  c(T)\|\nabla\du\|_{L_2(0,T;L_2)}.
 \end{aligned}$$
 Similarly,  
$$ \|\dA\, u^2\|_{L_4(0,T;L_2)}\\
 \leq \|t^{-1/2}\dA\|_{L_\infty(0,T;L_2)}\|t^{1/2}u^2\|_{L_4(0,T;L_\infty)},$$
 whence,  using  \eqref{eq:regularity},   \eqref{eq:www} and \eqref{l128} gives 
 \begin{equation}\label{l9} \| w\|_{L_4(0,T;L_2)} 
  \leq c(T)\|\nabla\du\|_{L_2(0,T;L_2)}.
\end{equation}
In order to bound $w_t,$  it suffices to derive an appropriate estimate 
in $L_{4/3}(0,T;L_{3/2})$ for 
$$(\dA\,u^2)_t=\dA\,u^2_t+(\dA)_t\,u^2.$$
Thanks to \eqref{eq:regularity} and \eqref{l128} 
$$
\begin{aligned}
\|\dA\,u^2_t\|_{L_{4/3}(0,T;L_{3/2})}&\leq \|t^{-1/2}\dA\|_{L_\infty(0,T;L_2)}\|t^{1/2}u^2_t\|_{L_{4/3}(0,T;L_6)}\\
&\leq c(T)\|\nabla\du\|_{L_2(0,T;L_2)}.
\end{aligned}
$$
One can bound the other term as follows: 
$$
\|\dA_t\,u^2\|_{L_{4/3}(0,T;L_{3/2})}\leq
\|\dA_t\|_{L_2(0,T\times \T^d)}\|u^2\|_{L_4(0,T;L_6)}.
$$
Differentiating \eqref{eq:dA} with respect to $t$ and using \eqref{eq:smallu} for $u^1$ and $u^2,$ 
we see that
$$
\|\dA_t\|_{2}\leq C\biggl(\|\nabla\du\|_{2}+\Big\|t^{-1/2}\int_0^t \nabla \du\,d\tau\Big\|_2
\bigl(\|t^{1/2}\nabla u^1\|_\infty+\|t^{1/2}\nabla u^2\|_\infty\bigr)\biggr)\cdotp
$$
Therefore 
$$\|\dA_t\|_{L_2(0,T\times \T^d)}\leq C\|\nabla\du\|_{L_2(0,T\times\T^d)},$$ and thus, owing to \eqref{eq:regularity}, 
$$\|\dA_t\,u^2\|_{L_{4/3}(0,T;L_{3/2})}\leq c(T)\|\nabla\du\|_{L_2(0,T\times\T^d)}.$$
Altogether, this  gives \eqref{eq:w}. 
\end{proof}

\medbreak
Next, let us restate the equations for $(\du,\dQ)$ as the following system for $(z,\dQ)$: 
\begin{equation}\label{l21}
 \begin{aligned}
&{\rho_0} z_t -\Delta_{u^1} z + \nabla_{u^1}\dQ = 
 (\Delta_{u^2}-\Delta_{u^1})u^2+(\nabla_{u^1}-\nabla_{u^2})Q^2
   - {\rho_0} w_t + \Delta_{u^1} w, \\
&\div_{\!u^1} z=0.
 \end{aligned}
\end{equation}
Let us test the equation by $z.$ We first notice the following  crucial property
thanks to which one does not have to care about the difference of the pressures:
\begin{equation}\label{l22}
 \int_{\T^d} (\nabla_{u^1} \dQ)\cdot z \, dx = -\int_{\T^d} \div_{\!u^1} z \, \dQ\, dx =0.
\end{equation}
So we have
\begin{equation}\label{l23}
 \frac 12 \frac{d}{dt} \int_{\T^d} {\rho_0} |z|^2 \,dx + \int_{\T^d} |\nabla_{u^1} z|^2 dx = \sum_{k=1}^4 I_k,
\end{equation}
where
$$
\begin{array}{lll}
 &I_1:= \Int_{\T^d} \bigl((\Delta_{u^2} -\Delta_{u^1}) u^2\bigr) \cdot z \,dx,  \quad 
&I_2:= \Int_{\T^d}  ((\nabla_{u^1}-\nabla_{u^2})Q^2)\cdot z \,dx,\\[1ex]
 &I_3:= - \Int_{\T^d} \, {\rho_0} \, w_t \cdot z \,dx, \quad &I_4: = \Int_{\T^d} (\Delta_{u^1} w)\cdot z\, dx.
\end{array}
$$
We have, using  \eqref{l4}, \eqref{eq:A} and \eqref{eq:smallu},
$$
\begin{aligned}
 I_1 &=  \int_{\T^d} \div\bigl((\dA\, {}^T\!A_2+ A_1 \,{}^T\!\dA)\nabla u^2\bigr)\cdot z\, dx  \\
 &\leq \int_{\T^d} \bigl|\dA\,{}^T\!A_2 + A_1\,\!{}^T\! \dA\bigr|\:  |\nabla u^2 |\: |\nabla z| \,dx \leq C \|t^{-1/2} \dA\|_2 \|t^{1/2} \nabla u^2\|_\infty \|\nabla z\|_2.
\end{aligned}
$$
Therefore, thanks to \eqref{eq:regularity} and \eqref{l128},
\begin{eqnarray}\label{l27}
 \int_0^T I_1(t)\,dt &&\!\!\!\!\!\!\!\!\!\!\leq C \|t^{-1/2} \dA\|_{L_\infty(0,T;L_2)}
  \|t^{1/2}\nabla u^2\|_{L_2(0,T;L_\infty)}  \|\nabla z\|_{L_2(0,T\times\T^d)}\nonumber\\
  &&\!\!\!\!\!\!\!\!\!\!\leq c(T)  \|\nabla\du\|_{L_2(0,T\times\T^d)}  \|\nabla z\|_{L_2(0,T\times\T^d)}.
\end{eqnarray}
Next, 
\begin{equation}\label{l29}
 I_2 \leq  \bigg|\int_{\T^d} \dA \nabla Q^2 \cdot z \,dx \bigg|\leq C\|t^{-1/2} \dA\|_2 \|t^{1/2} \nabla Q^2\|_4 \|z\|_4, 
\end{equation}
whence, according to \eqref{l128} and Sobolev embedding
$$
\begin{aligned}
 \int_0^T I_2(t)\,dt &\leq \|t^{-1/2}\dA\|_{L_\infty(0,T;L_2)}  \|t^{1/2} \nabla Q^2\|_{L_2(0,T;L_4)}
  \|z\|_{L_2(0,T;L_4)}\\
&\leq c(T) \|\nabla \du\|_{L_2(0,T\times\T^d)}  \|z\|_{L_2(0,T;H^1)}.
 \end{aligned}
 $$
 At this stage, one may use Lemma \ref{l:poincare} to bound $z$ in $H^1$: 
 we get, for some constant $C$ depending only on $\rho_0,$
 \begin{equation}\label{eq:zH1}
  \|z\|_{H^1}\leq C\bigl(\|\sqrt{{\rho_0}} z\|_2+\|\nabla z\|_2\bigr).
  \end{equation}
    Therefore,  one  concludes  that
 $$
  \int_0^T I_2(t)\,dt  \leq c(T)\left(\|\sqrt{\rho_0}\, z\|_{L_\infty(0,T;L_2)}
  +\|\nabla z\|_{L_2(0,T\times\T^d)}\right)
  \|\nabla \du\|_{L_2(0,T\times\T^d)}\cdotp
 $$
 Next, using H\"older inequality, one can write that
 $$\int_0^T I_3(t) \, dt \leq \|{\rho_0}\|_{\infty}^{3/4} \|w_t\|_{L_{4/3}(0,T;L_{3/2})} \|{\rho_0}^{1/4}z\|_{L_{4}(0,T;L_{3})}.$$
Note that  from  H\"older inequality and the Sobolev embedding $H^1(\T^d)\hookrightarrow L_6(\T^d),$ we have
$$
 \|{\rho_0}^{1/4}z\|_{L_{4}(0,T;L_{3})}\leq \|\sqrt{\rho_0}\, z\|_{L_\infty(0,T;L_2)}^{1/2}\|z\|_{L_2(0,T;L_6)}^{1/2}
 \leq C \|\sqrt{\rho_0}\, z\|_{L_\infty(0,T;L_2)}^{1/2}\|z\|_{L_2(0,T;H^1)}^{1/2}.$$
   Then,  taking advantage of \eqref{eq:zH1} and \eqref{eq:w}, we conclude that 
  $$
 \int_0^T\!\! I_3(t) \, dt \leq c(T)\left(\|\sqrt{\rho_0}\, z\|_{L_\infty(0,T;L_2)}
  \!+\!\|\nabla z\|_{L_2(0,T\times\T^d)}\right)^{\!1/2}\! \|\sqrt{\rho_0}\, z\|_{L_\infty(0,T;L_2)}^{1/2}
  \|\nabla \du\|_{L_2(0,T\times\T^d)}\cdotp
  $$
Finally, integrating by parts, and using \eqref{eq:w} and \eqref{eq:smallu},
$$\begin{aligned}
 \int_0^T I_4(t)\,dt &\leq \int_0^T  \int_{\T^d} |\nabla_{u^1} w|\, | \nabla_{u^1} z| \,dx \, dt\\
  &\leq    \frac12\int_0^T \|\nabla_{u^1} z\|_2^2\, dt+ \frac12\int_0^T \|\nabla_{u^1} w\|_2^2\, dt\\
   &\leq   \frac12\int_0^T \|\nabla_{u^1} z\|_2^2\, dt+c(T)\int_0^T \|\nabla \du\|_2^2\,dt.
\end{aligned}$$
So altogether, this gives for all small enough $T>0,$
\begin{equation}\label{l33}
\sup_{t\in[0,T]} \|\sqrt{{\rho_0}} z(t)\|_2^2 + \int_0^T \|\nabla z\|_2^2\, dt  \displaystyle \leq  c(T)\int_0^T \|\nabla \du\|^2_2 \,dt.
\end{equation}
Combining with \eqref{eq:w}, we conclude that
$$
\int_0^T \|\nabla \du\|^2_2 \,dt\leq c(T)\int_0^T \|\nabla \du\|^2_2 \,dt.
$$
Hence $\nabla\du\equiv0$ on $[0,T]\times\T^d$ if $T$ is small enough.
\medbreak
Then, plugging that information in \eqref{l33} yields 
$$\|\sqrt{\rho_0}\, z\|_{L_\infty(0,T;L_2)}+\|\nabla z\|_{L_2(0,T\times\T^d)}=0.$$ 
Combining with
Lemma \ref{l:poincare} finally implies that $z\equiv0$ 
on $[0,T]\times\T^d,$ and \eqref{eq:w} clearly  yields $w\equiv0.$
 Therefore we proved that for small enough $T>0,$ 
 $$ u^1\equiv u^2\quad\hbox{on }\ [0,T]\times\T^d.$$ 
 Reverting to Eulerian coordinates, we conclude that  the two solutions of $(INS)$  coincide 
 on $[0,T]\times\T^d.$ 
  Then  standard connectivity arguments yield uniqueness on the whole $\R_+.$


\section{Appendix}

We here establish  the weighted Poincar\'e inequalities and  
the   Desjardins interpolation inequality \eqref{d0} that have been used several times in the proof of existence,  
and    results  for the `twisted' divergence equation that were the key to the proof of uniqueness. 
\medbreak
Let us start with the Poincar\'e inequalities in the case where $\Omega$ is the unit torus $\T^d.$ 
\begin{lem}\label{l:poincare}
Let $a:(0,1)^d\to\R$ be a nonnegative and nonzero measurable function. 
Then we have for all $z$ in $H^1(\T^d),$ 
\begin{equation}\label{eq:poincare}
\|z\|_{2}\leq \frac1M\biggl|\int_{\T^d}az\,dx\biggr|+\biggl(1+\frac1M\|M-a\|_2\biggr)\|\nabla z\|_2
\quad\hbox{with}\quad M:=\int_{\T^d} a\,dx.
\end{equation}
Furthermore, in dimension $d=2,$  there exists an absolute constant $C$ so that 
\begin{equation}\label{eq:limitpoincare}
\|z\|_{2}\leq    \frac1M\biggl|\int_{\T^2}az\,dx\biggr|+C\log^{\frac12}\biggl(e+\frac{\|a-M\|_{2}}{M}\biggr)\|\nabla z\|_{2}.
 \end{equation}
\end{lem}
\begin{proof}
To prove \eqref{eq:poincare}, we start with the obvious inequality
\begin{equation}\label{eq:z}\|z\|_2\leq|\bar z|+\|\nabla z\|_2\quad\hbox{with}\quad\bar z:=\int_{\T^d}z\,dx,\end{equation}
then we use the fact that
\begin{equation}\label{eq:Mz}
M\bar z=\int_{\T^d} az\,dx +\int_{\T^d}(M-a)(z-\bar z)\,dx,
\end{equation}
which, according to  the classical Cauchy-Schwarz and Poincar\'e inequalities in the torus, implies that 
\begin{equation}\label{eq:mean}
M|\bar z|\leq \biggl|\int_{\T^d} az\,dx\biggr| +\|M-a\|_2\|\nabla z\|_2.
\end{equation}
To prove Inequality \eqref{eq:limitpoincare},  we decompose $\wt z:=z-\bar z$ 
 into Fourier series:
$$
\wt z(x)=\sum_{k\in\Z^2\setminus\{(0,0)\}} \wh z_k\, e^{2i\pi k\cdot x}.
$$
Then, for any integer $n,$ we  set
$$
\wt z_n(x):=\sum_{1\leq|k|\leq n}  \wh z_k\, e^{2i\pi k\cdot x}.
$$
 Because the average of $\wt z_n$ is $0,$   one may write  for any  integer $n,$
$$
\int_{\T^2}(a -M)\wt z\,dx=\int_{\T^2}   a\wt z_n\,dx+\int_{\T^2}(a-M)  (\wt z-\wt z_n)\,dx.
$$
Therefore, using H\"older inequality,
\begin{equation}\label{eq:p1}
\biggl| \int_{\T^2} (a -M)\wt z\,dx\biggr|\leq M\|\wt z_n\|_{\infty}
+\|a-M\|_{2}\|\wt z-\wt z_n\|_{2}.
\end{equation}
By Cauchy-Schwarz inequality, we have for all $x\in\T^2,$
\begin{eqnarray}\label{eq:des2}
| \wt z_n(x)|&&\!\!\!\!\!\!\leq \sum_{1\leq|k|\leq n}|2\pi k\wh z_k|\, \frac{|e^{2i\pi k\cdot x}\bigr|}{2\pi|k|}\nonumber\\
&&\!\!\!\!\!\!\leq \|\nabla z\|_{2}\biggl(\sum_{1\leq |k|\leq n} \frac1{4\pi^2|k|^{2}}\biggr)^{1/2}\nonumber\\
&&\!\!\!\!\!\!\leq C\sqrt{\log n} \, \|\nabla z\|_{2}.
\end{eqnarray}
Inserting that inequality in \eqref{eq:p1} yields 
$$\biggl| \int_{\T^2} (a -M)\wt z\,dx\biggr|\leq C\Bigl(\sqrt {\log n}\:M+n^{-1}\|a-M\|_{2}\Bigr)\|\nabla z\|_{2}.
$$
Then taking $n\approx\|a-M\|_{2}/M$ and mingling with 
\eqref{eq:z} and \eqref{eq:Mz} gives \eqref{eq:limitpoincare}.
\end{proof}
\medbreak
Let us now prove   Lemma \ref{DLest}.  As above, it  is based on decomposition into low 
and high frequency, while the  original proof by B. Desjardins in \cite{Des-CPDE} relies on Trudinger inequality. 
The starting point is that for any  $n\in\N^*,$  we have
$$\begin{aligned}
\int_{\T^2}\rho z^4\,dx &=\int_{\T^2} \rho z^2 \bigl(\bar z+ \wt z_n+(\wt z-\wt z_n)\bigr)^2\,dx\\&\leq
3\biggl(\bar z^2\|\sqrt\rho z\|_{2}^2+\|\wt z_n\|_{\infty}^2\|\sqrt\rho z\|_{2}^2
+\sqrt{\rho^*}\|\wt z-\wt z_n\|_{4}^2\biggl(\int_{\T^2}\rho z^4\,dx\biggr)^{\frac12} \biggr)\cdotp
\end{aligned}
$$
We thus have, using Young inequality and embedding $\dot H^{\frac12}\hookrightarrow L^4,$
\begin{equation}\label{eq:des1}
\int_{\T^2}\rho z^4\,dx \leq 6\|\sqrt\rho z\|_{2}^2\bigl(\bar z^2+\|\wt z_n\|_{\infty}^2\bigr)
+C\rho^*\|\wt z-\wt z_n\|_{\dot H^{\frac12}}^4.
\end{equation}
Obviously,  we have
\begin{equation}\label{eq:des3}
\|\wt z-\wt z_n\|_{\dot H^{\frac12}}\leq \frac1{\sqrt n}\|\nabla z\|_{2}.\end{equation}
Hence, using \eqref{eq:des2} to bound the first term of the right-hand side of \eqref{eq:des1},  we get
$$
\int_{\T^2}\rho z^4\,dx \lesssim \bar z^2\|\sqrt\rho z\|_{2}^2+
\bigl(\log n \|\sqrt\rho z\|_{2}^2+ n^{-2}\rho^*\|\nabla z\|_{2}^2\bigr)\|\nabla z\|_{2}^2.
$$
Taking for $n$ the closest positive integer to $\frac{\sqrt{\rho^*}\|\nabla z\|_{2}}{\|\sqrt \rho z\|_{2}},$ we end up with 
\begin{equation}\label{eq:interpopo}
\biggl(\int_{\T^2} \rho z^4\,dx\biggr)^{\frac12}\leq C \|\sqrt\rho z\|_{2}
\biggl(|\bar z|+\|\nabla z\|_{2}\log^{\frac12}\biggl(e+\frac{\rho^*\|\nabla z\|^2_{2}}{\|\sqrt \rho z\|^2_{2}}\biggr)\biggr)\cdotp
\end{equation}
Then, bounding  $\bar z$ as in the proof of Inequality \eqref{eq:limitpoincare} yields  \eqref{d0}.
\qed
\medbreak
Inequality \eqref{eq:interpopo} (and thus Lemma \ref{DLest}) may be easily adapted to any bounded domain $\Omega.$ 
Indeed,  as the estimates therein  are invariant by translation, one may assume with 
no loss of generality that  $\Omega\subset(0,R)^d$ for some $R>0.$  Now, any function in $H_0^1(\Omega)$ may 
be extended  by $0$ to a function of $H_0^1((0,R)^d),$ thus  to $H^1$ on the torus with size $R.$ 
Then one can apply the above results, and get   the same inequalities for  some constant depending only on $R.$
In order to track the dependency with respect to $R,$ it suffices to make 
 the change of function $z'(x):=z(Rx)$ and to  apply to $z'$ the inequality for some domain included in $(0,1)^d.$
For example, Inequality \eqref{eq:poincare} becomes  for all $z$ in $H^1_0(\Omega)$ with $\Omega$ of diameter less than $R,$ 
$$
\|z\|_{2}\leq \frac{R^{d/2}}{M}\biggl|\int_{\T^d}az\,dx\biggr|+R\biggl(1+\frac{R^{d/2}}M\|M-a\|_2\biggr)\|\nabla z\|_2.
$$

\medbreak
Finally,  let us  consider the twisted divergence equation. 
The following lemma  (in the spirit  of the corresponding result in \cite{DM-luminy}) is the key to the proof of uniqueness, in the torus case.
\begin{lem}\label{l:div} Let $A$ be a matrix valued function on $[0,T]\times\T^d$ 
satisfying 
\begin{equation}\label{eq:detA}
\det A\equiv1.
\end{equation}
There exists a  constant $c$ depending only on $d$, such that if 
 \begin{equation}\label{eq:smallA}
 \|\Id-A\|_{L_\infty(0,T;L_\infty)}+\|A_t\|_{L_2(0,T;L_6)}\leq c
 \end{equation}
 then for all function $g:[0,T]\times\T^d\to\R$ satisfying 
$g\in L^2(0,T\times\T^d)$ and 
$$ g=\div R\ \hbox{ with }\ R\in L_4(0,T;L_2)\ \hbox{ and }\ R_t\in L_{4/3}(0,T;L_{3/2}),
$$
the equation 
$$\div(Aw)=g \quad\hbox{in}\quad [0,T]\times \T^d$$
admits a solution $w$ in the  space 
$$X_T:=\Bigl\{v\in L_4(0,T;L_2(\T^d))\:, \: \nabla v\in L_2(0,T;L_2(\T^d))
\  \hbox{and}\ v_t\in L_{4/3}(0,T;L_{3/2}(\T^d))\Bigr\}
$$
satisfying the following  inequalities for some  constant $C=C(d)$: 
\begin{equation}\label{eq:diveq}\begin{array}{c}
\|w\|_{L_4(0,T;L_2)}\leq C\|R\|_{L_4(0,T;L_2)}, \quad
\|\nabla w\|_{L_2(0,T;L_2)}\leq C\|g\|_{L_2(0,T;L_2)}\\[1ex] \hbox{\rm and }\ 
\|w_t\|_{L_{4/3}(0,T;L_{3/2})}\leq C\|R\|_{L_4(0,T;L_{2})}+C\|R_t\|_{L_{4/3}(0,T;L_{3/2})}.\end{array}\end{equation}
\end{lem}
\begin{proof}
Let  $g$ satisfy the conditions of the lemma. 
Then for any $v\in X_T,$ we set 
$$\displaylines{\quad\Phi(v):=\nabla\Delta^{-1}\div\bigl((\Id-A)v+R-\overline{(\Id-A)v+R}\bigr)\hfill\cr\hfill\quad\hbox{with }\ 
\overline{(\Id-A)v+R}:=\frac{1}{|\T^d|}\int_{\T^d}\bigl( (\Id-A)v+R\bigr)\,dx.\quad}
$$
It is obvious that $\Phi(v)$  satisfies the linear equation
$$
\div w=\div((\Id-A)v)+g\qquad\hbox{in }\ (0,T)\times\T^d.
$$
Furthermore the operator $\nabla\Delta^{-1}\div$
maps the set $L_{2,0}(\T^d)$ (of $L_2(\T^d)$ functions with zero average) to itself (and with norm $1$),  and  
$\bigl|\overline{(\Id-A)v+R}\bigr|\leq \|(\Id-A)v+R\|_2.$ Hence we have
$$\begin{aligned}
\|\Phi(v)\|_{L_4(0,T;L_2)}&\leq 2\|(\Id-A)v+R\|_{L_4(0,T;L_2)}\\
&\leq 2\|\Id-A\|_{L_\infty(0,T;L_\infty)}\|v\|_{L_4(0,T;L_2)}+2\|R\|_{L_4(0,T;L_2)}.\end{aligned}
$$
Then using the fact that \eqref{eq:detA} implies that (see e.g. the Appendix of \cite{DM-cpam})
\begin{equation}\label{eq:magictrick}\div(Az)=A^{T}:\nabla z,\end{equation}
we get 
$$\begin{aligned}
\|\nabla\Phi(v)\|_{L_2(0,T;L_2)}&\leq \|{}^T\!(\Id-A):\nabla v\|_{L_2(0,T;L_2)}
+\|g\|_{L_2(0,T;L_2)}\\&\leq    \|\Id-A\|_{L_\infty(0,T;L_\infty)}\|\nabla v\|_{L_2(0,T;L_2)}+\|g\|_{L_2(0,T;L_2)}.
\end{aligned}$$
And finally, because $((\Id-A)v)_t=(\Id-A)v_t-A_tv,$ we have
$$\begin{aligned}
\|(\Phi(v))_t\|_{L_{4/3}(0,T;L_{3/2})}&\leq C\bigl(\|(\Id-A)v_t\|_{L_{4/3}(0,T;L_{3/2})}
\!+\!\|A_tv\|_{L_{4/3}(0,T;L_{3/2})}\!+\!\|R_t\|_{L_{4/3}(0,T;L_{3/2})}\bigr)\\
&\leq   C\bigl(\|\Id-A\|_{L_\infty(0,T;L_\infty)}\|v_t\|_{L_{4/3}(0,T;L_{3/2})}
\\&\hspace{2.5cm}+\|A_t\|_{L_2(0,T;L_6)}\|v\|_{L_4(0,T;L_2)}+\|R_t\|_{L_{4/3}(0,T;L_{3/2})}\bigr)\cdotp
\end{aligned}
$$
This proves that $\Phi$ maps $X_T$ to $X_T.$ Then, obvious variations on the above computations
give for any couple $(v_1,v_2)$ in $X^2_T,$  if $c$ in  \eqref{eq:smallA} is small enough,
$$
\|\Phi(v_2)-\Phi(v_1)\|_{X_T}\leq  C\bigl(\|\Id-A\|_{L_\infty(0,T;L_\infty)}+\|A_t\|_{L_2(0,T;L_6)}\bigr)
\|v_2-v_1\|_{X_T}\leq \frac12\|v_2-v_1\|_{X_T}.
$$
Hence, applying the standard Banach fixed point theorem in $X_T$ provides 
a  solution to the equation $\Phi(v)=v.$ Then looking back at the above 
computations in the case $\Phi(v)=v$ gives the desired inequalities. 
\end{proof}
In the bounded domain case, the previous lemma can be adapted as follows.
\begin{lem}\label{l:div-dom} Let $\Omega$ be a $\cC^2$ bounded domain of $\R^d,$ and  $A,$  a matrix valued function on $[0,T]\times\Omega$ 
satisfying  \eqref{eq:detA}. If  \eqref{eq:smallA} is fulfilled  
 then for all function $R:[0,T]\times\Omega\to\R^d$ satisfying 
$\div R\in L^2(0,T\times\Omega),$ $R\in L_4(0,T;L_2),$ $R_t\in L_{4/3}(0,T;L_{3/2})$ and $R\cdot n\equiv0$ on $(0,T)\times\d\Omega,$
the equation 
$$\div(Aw)=\div R=:g \quad\hbox{in}\quad [0,T]\times \Omega$$
admits a solution in the  space 
$$X_T:=\Bigl\{v\in L_2(0,T;H^1_0(\Omega))\:, \:  v\in L_4(0,T;L_2(\Omega))
\  \hbox{and}\ v_t\in L_{4/3}(0,T;L_{3/2}(\Omega))\Bigr\},
$$
that satisfies Inequalities \eqref{eq:diveq}.
\end{lem}
\begin{proof}
This is essentially Theorem 4.1 of \cite{DM-luminy} adapted to time dependent coefficients.
 Recall that Theorem 3.1
 of \cite{DM-luminy} states that there exists a linear operator $\cB:k\to u$ that is continuous on 
 $L_p(\Omega;\R^d)$    (for all  $1<p<\infty$) 
so that for all $k$ in $L_p(\Omega;\R^d)$ 
the vector field $u$ satisfies  
\begin{equation}\label{eq:DIV}\int_\Omega u\cdot\nabla\phi\,dx=\int_\Omega k\cdot\nabla\phi\,dx\quad\hbox{for all }\ \phi\in\cC^\infty(\overline\Omega;\R),\end{equation}
and such that, if in addition $\div k$ is in $L_p(\Omega)$ and $k\cdot n|_{\d\Omega}=0$  then 
$u$ is in $W^1_{p,0}(\Omega)$ with
$$\|u\|_{W^1_p(\Omega)}\leq C\|\div k\|_{L_p(\Omega)}.$$

Then we define $\Phi$ on the set $X_T$ (treating the time variable as a parameter) by
\begin{equation}\label{eq:B}
\Phi(v):=\cB\bigl((\Id-A)v +R\bigr).
\end{equation}
The above result ensures that 
$$\begin{aligned}
\|\nabla\Phi(v)\|_{L_2(0,T;L_2)}&\lesssim\|{}^T\!(\Id-A):\nabla v+g\|_{L_2(0,T;L_2)}\\
&\lesssim\|g\|_{L_2(0,T;L_2)}+\|\Id-A\|_{L_\infty(0,T;L_\infty)}\|\nabla v\|_{L_2(0,T;L_2)},
\end{aligned}$$
as well as 
$$\begin{aligned}
\|\Phi(v)\|_{L_4(0,T;L_2)}&\lesssim  \|(\Id-A)v +R\|_{L_4(0,T;L_2)}\\&\lesssim\|\Id-A\|_{L_\infty(0,T;L_\infty)}\|v\|_{L_4(0,T;L_2)}+R\|_{L_4(0,T;L_2)}.
\end{aligned}$$
Moreover,  differentiating \eqref{eq:B} with respect to time gives $(\Phi(v))_t=\cB((\Id-A)v_t-A_t v+R_t)$
whence, using the continuity property of $\cB$ on $L_{3/2}(\Omega),$ then performing a time integration, 
$$\begin{aligned}
\|(\Phi(v))_t\|_{L_{4/3}(0,T;L_{3/2})} &\lesssim \|(\Id-A)v_t-A_t v\|_{L_{4/3}(0,T;L_{3/2})}+\|R_t\|_{L_{4/3}(0,T;L_{3/2})}\\[1ex]
&\lesssim\|\Id-A\|_{L_\infty(0,T;L_\infty)}\|v_t\|_{L_{4/3}(0,T;L_{3/2})}\\&\hspace{3.4cm}+\|A_t\|_{L_2(0,T;L_6)}\|v\|_{L_4(0,T;L_2)}+\|R_t\|_{L_{4/3}(0,T;L_{3/2})},\end{aligned}
$$
and one can  conclude by the Banach fixed point theorem, exactly as in  Lemma \ref{l:div}. 
\end{proof}

\noindent{\bf Acknowledgments.} This work was partially supported by the Simons - Foundation grant 346300 and the Polish Government MNiSW 2015-2019 matching fund.
The first author (R.D.)  is also  supported by ANR-15-CE40-0011.
The second author (P.B.M.) has been partly supported by National  Science  Centre  grant 2014/14/M/ST1/00108 (Harmonia).

\end{document}